\documentclass{amsart}
\usepackage{amssymb,enumerate,amsmath,amscd}
\usepackage[svgnames]{xcolor}
\usepackage[colorlinks,bookmarks]{hyperref}
\usepackage{bm}
\usepackage{verbatim}
\hypersetup{
    colorlinks = true,
    citecolor = teal,
    linkcolor = purple,
    anchorcolor = blue,
    filecolor = blue,
   urlcolor = blue
}

\usepackage{tikz,tikz-cd}
\usetikzlibrary{arrows,positioning,shapes,decorations.pathreplacing}

\usepackage[mathscr]{euscript}
\newcommand{\goaway}[1]{}

\usepackage[margin=1.2in]{geometry}

\usepackage{cleveref}
\usepackage{subcaption}

\usepackage{enumitem}
\newlist{deflist}{enumerate}{1}
\setlist[deflist]{label=(\arabic{deflisti}), ref=\thelemma.(\arabic{deflisti}),noitemsep}

\newtheorem{lemma}{Lemma}[subsection]
\newtheorem{theorem}[lemma]{Theorem}
\newtheorem{proposition}[lemma]{Proposition}
\newtheorem{corollary}[lemma]{Corollary}

\newtheorem*{ack}{Acknowledgements}

\theoremstyle{definition}
\newtheorem{example}[lemma]{Example}
\newtheorem{remark}[lemma]{Remark}

\newtheorem{definition}[lemma]{Definition}
\newtheorem{d-a}[lemma]{Definition-Assumption}

\renewcommand{\P}{\mathscr{P}}      
\newcommand{\zero}{\hat{0}}			
\newcommand{\A}{\mathscr{A}}		
\newcommand{\B}{\mathscr{B}}        
\newcommand{\pr}{\bar{p}}           
\newcommand{\G}{\mathbb{G}}         
\newcommand{\Z}{\mathbb{Z}}			
\newcommand{\R}{\mathbb{R}}			
\newcommand{\C}{\mathbb{C}}         
\newcommand{\QQ}{\mathbb{Q}}        
\newcommand{\proj}{\mathrm{pr}}
\newcommand{\ii}{\bm{\iota}}

\renewcommand{\k}[1][k]{\mathbf{#1}}         
\newcommand{\kplus}[1][k]{(\k[#1],1)}
\newcommand{\SConf}[1]{\Conf^{#1}}
\newcommand{\IL}{\Gamma}            
\newcommand{\X}{\mathfrak{X}}       

\newcommand{\coefmap}{\mathbf{a}}    
\newcommand{\rootmap}{\mathbf{b}}    

\DeclareMathOperator{\Hom}{Hom}
\DeclareMathOperator{\Conf}{Conf}

\DeclareMathOperator{\id}{id}

\DeclareMathOperator{\Aut}{Aut}
\DeclareMathOperator{\Der}{Der}
\DeclareMathOperator{\ad}{ad}
\DeclareMathOperator{\secat}{secat}

\newcommand{\pg}[1]{\Sigma_{#1}} 

\newcommand{\clie}{\mathcal{L}}
\newcommand{\lie}{\mathfrak{h}}
\newcommand{\ab}{\mathfrak{a}}
\newcommand{\CC}{\mathscr{C}}        
\newcommand{\LL}{\mathbb{L}}
\newcommand{\EE}{\mathsf{E}}
\newcommand{\II}{\mathsf{I}}
\newcommand{\ns}{\mathsf{n}}
\newcommand{\zs}{\mathsf{z}}
\newcommand{\ps}{\mathsf{p}}
\newcommand{\as}{\mathsf{a}}
\newcommand{\bs}{\mathsf{b}}
\newcommand{\ws}{\mathsf{w}}
\newcommand{\xs}{\mathsf{x}}
\newcommand{\ys}{\mathsf{y}}
\newcommand{\us}{\mathsf{u}}
\newcommand{\vs}{\mathsf{v}}

\newcommand{\tc}{\mathsf{TC}}

\newcommand{\ww}{\overline{\omega}}

\makeatletter
\newcommand{\mylabel}[2]{#2\def\@currentlabel{#2}\label{#1}}
\makeatother
\makeatletter
\newcommand\mynobreakpar{%
\par\nobreak\@afterheading} 
\makeatother

\title{
Monodromy of supersolvable toric arrangements
}
\author{Christin Bibby}
\thanks{Bibby supported by NSF DMS-2204299}
\address{Louisiana State University, Baton Rouge, LA, USA}
\email{\url{bibby@math.lsu.edu}}
\author{Daniel C.\ Cohen}
\address{Louisiana State University, Baton Rouge, LA, USA}
\email{\url{cohen@math.lsu.edu}}
\author{Emanuele Delucchi}
\address{University of Applied Arts and Sciences of Southern Switzerland}
\email{\url{emanuele.delucchi@supsi.ch}}



\begin{document}

\begin{abstract}
We study topological aspects of supersolvable abelian arrangements, toric arrangements in particular. The complement of such an arrangement sits atop a tower of fiber bundles, and we investigate the relationship between these bundles and bundles involving classical configuration spaces. In the toric case, we show that the monodromy of a supersolvable arrangement bundle factors through the Artin braid group, and that of a strictly supersolvable arrangement bundle factors further through the Artin pure braid group. The latter factorization is particularly informative -- we use it to determine a number of invariants of the complement of a strictly supersolvable arrangement, including the cohomology ring and the lower central series Lie algebra of the fundamental group.
\end{abstract}

\maketitle

\tableofcontents

\vspace*{14pt}
\section{Introduction}

\subsection{Background}
Over the last decades, the study of complements of hyperplane arrangements in complex vector spaces has given rise to a rich theory at the crossroads of algebraic topology and combinatorics. One of the seminal papers in this field is the work of Arnol'd on the cohomology of pure braid groups \cite{arnold}, motivated by the connection to configuration spaces and the classical Fadell--Neuwirth theorem \cite{FN}. 
In this sense, complements of hyperplane arrangements and their fundamental groups are generalizations of configuration spaces of ordered points in the plane and pure braid groups. 

This analogy is particularly strong for {\em fiber-type} arrangements, introduced by Falk and Randell \cite{FR} as the class of hyperplane arrangements satisfying a recursive fibration property akin to Fadell and Neuwirth's for configuration spaces. Fiber-type arrangements of hyperplanes have been in the focus of substantial research: they can be characterized purely combinatorially via Stanley's theory of supersolvable lattices \cite{stanley}, and much of the theory of braid groups and configuration spaces has an analogue in this more general context. For instance, we mention results on the lower central series (LCS) of the fundamental group of the complement \cite{FR} and the associated LCS Lie algebra \cite{cohen,CCX}, isomorphic to the holonomy Lie algebra of the arrangement \cite{KohnoHolo}. A key fact in this context, first observed by Cohen \cite{cohen}, is that the fiber bundles arising in the hyperplane arrangement case can be pulled back from classical Fadell--Neuwirth bundles for configuration spaces of points in the plane. 
This facilitates the explicit computation of the monodromy of (fiber-type) arrangement bundles \cite{CSmonodromy,cohen}, and the determination of the cohomology ring of the complement from the iterated semidirect product  
structure of its fundamental group \cite{DCalmostdirect}.

Recently, the focus of the theory of arrangements has been broadened towards the case of hypersurfaces in complex tori ({\em toric arrangements}) and, more generally, in connected abelian Lie groups ({\em abelian arrangements}). This research direction has gained substantial momentum from the 2010's
in the wake of De Concini, Procesi and Vergne's seminal work  on vector partition functions and Dahmen-Micchelli spaces of splines \cite{DCP,DPV,DPV2}, among others. 
Some notable advances have been made on the topological side, including the computation of the integer cohomology ring in the toric case \cite{CDDMP} and in the non-compact abelian case \cite{BPP}. Such topological invariants appear to be strongly related to the structure of the partially ordered set of connected components of intersections of the hypersurfaces (the so-called {\em poset of layers}) which, in turn, has been studied from the combinatorial point of view -- see, for instance, \cite{Tomz,ER,DR,bibby}.

The notion of fiber-type arrangements in the toric and abelian setting has been introduced in \cite{BD}, together with an equivalent combinatorial characterization that generalizes Stanley's supersolvability for lattices.  A main takeaway from \cite{BD} is that in this broader context there are two combinatorial notions of supersolvability: one is equivalent to the inductive fibration property for the arrangement complement and the other, stronger one (called {\em strict supersolvability}) defines a class of posets where closer analogues of the features of classical supersolvable lattices hold. 
While a thorough poset-theoretic investigation of this circle of ideas, leading to an even finer classification, has been carried out in \cite{PPTV},
a main motivation of the current article is to carry out a further investigation from the topological point of view.

In \cite[Theorem 5.3.1]{BD}, it was noted that strict supersolvability of the poset of layers of an arrangement implies that the corresponding fiber bundles are  
pulled back from Fadell--Neuwirth  
bundles for suitable configuration spaces.
This raises two natural questions. 
First, are  the fiber bundles arising from the ``weaker'' notion of supersolvability realizable as pullbacks of configuration space bundles? 
Moreover, one can ask whether, at least in the special case of toric arrangements, invariants such as the monodromy, the cohomology ring, and the LCS Lie algebra can be determined by utilizing the aforementioned relationship between strict supersolvability and classical configuration spaces.

\subsection{Overview and structure of the paper} We further the topological study of supersolvable toric and abelian arrangements along the two directions mentioned above.

In \Cref{sec:arr}, we lay the foundations and 
show that the fiber bundles associated to any supersolvable abelian arrangement can be pulled back from Fadell-Neuwirth-type bundles involving orbit spaces of the action of products of symmetric groups on classical ordered configuration spaces (\Cref{thm:pullback}).

In \Cref{sec:toric}, we specialize to toric arrangements, where
the pullbacks are from spaces of configurations of points in the plane. In \Cref{thm:toricpullback}, we give a characterization of the maps along which the configuration space bundles are pulled back. 
These maps, a {\em coefficient map} $\k[a]$ into an unordered configuration space in the supersolvable case, and a {\em root map} $\k[b]$ into an ordered configuration space in the strictly supersolvable case, 
may be used to describe the monodromy of the fiber bundles associated with supersolvable toric arrangements. In particular, this monodromy factors through the Artin representation of the braid group in the automorphism group of the free group (\Cref{prop:coef/root}). As a consequence, we show that the fundamental group of the complement of \emph{any} supersolvable toric arrangement is an iterated semidirect product of free groups (\Cref{cor:almostdirect}), structure previously observed in the strictly supersolvable case in 
\cite{BD}. 
These results provide a clear distinction between supersolvable and strictly supersolvable toric arrangements. In the former case, the iterated semidirect product structure of the fundamental group is determined by braid automorphisms. In the latter, the monodromy factors further through the pure braid group, yielding almost-direct product structure in the sense of \cite{FR}.
 
In \Cref{sec:ssta}, we focus on the special case of {\em strictly} supersolvable toric arrangements. 
The complement of such an arrangement sits atop a tower of bundles, determined by a sequence of root maps to ordered configuration spaces. The main gist is that this sequence of root maps determines the structure of both the cohomology ring of the complement (\Cref{thm:H*ring}) and the LCS Lie algebra of its fundamental group (\Cref{thm:LCSliealg}). Specifically, each relevant root map induces a map in (first) homology, which we call a \emph{homological root homomorphism}. These homomorphisms, computed in terms of the defining characters of the arrangement in \Cref{thm:hrm}, may be used to obtain explicit presentations for both the cohomology ring and the LCS Lie algebra. The resulting cohomology presentation, different than those of \cite{CD,CDDMP,BPP}, exhibits the Koszulity of the 
cohomology algebra (\Cref{cor:koszul}).  We also compute the topological complexity of the complement, noting that it only depends on the ambient dimension and the {rank} of the arrangement (\Cref{thm:TC}).

Illustrations via concrete examples are provided throughout the paper. These include a family of rank two strictly supersolvable toric arrangements consisting of three hypersurfaces in the two-dimensional torus discussed in \Cref{subsec:rank2circ}, and the family of Weyl type C toric arrangements of arbitrary rank studied in \Cref{sec:typeC}. 
Presentations of the cohomology ring of the complement and the LCS Lie algebra of its fundamental group are obtained for both families.  
In rank two, we also demonstrate how our methods yield explicit fundamental group presentations.

\section{Arrangements and configuration spaces}\label{sec:arr}

\subsection{Abelian and toric arrangements}
Let $\G$ be a connected abelian Lie group, $\IL\cong\Z^d$ a finitely generated free abelian group, and $T=\Hom(\IL,\G)\cong\G^d$.

\begin{definition}
An \textbf{abelian arrangement} $\A$ is, for some finite set $\X=\X(\A)\subseteq\IL$, the collection of connected components of the subspaces
\[H_{\chi} := \{t\in T \mid \chi\in \ker(t) \}\] 
with $\chi\in\X(\A)$.

The \textbf{complement} of $\A$ is denoted by 
\[M(\A) := T\smallsetminus \bigcup_{\chi\in\X(\A)} H_\chi.\]

The \textbf{poset of layers} of $\A$ is the set $\P(\A)$ whose elements are the nonempty connected components of intersections $\cap_{\chi\in S}H_\chi$ where $S\subseteq\X(\A)$, partially ordered by reverse inclusion.
\end{definition}

\begin{remark} \label{rem:essential}
We pay special attention to two cases: when $\G=\C$, $T$ is a complex affine space and $\A$ is called a \textbf{hyperplane arrangement}; when $\G=\C^\times$,  $T$ is a complex torus and $\A$ is called a \textbf{toric arrangement}. 
We focus primarily on toric arrangements that are \textbf{essential}, i.e., where 
the maximal elements of $\P(\A)$ are points,
since as noted in \cite[Remark 2.7]{CDDMP} one can always find an essential arrangement $\A'$ in a torus $(\C^\times)^r$ such that $M(\A)\cong M(\A')\times(\C^\times)^{d-r}$. We refer to $r$ as the \textbf{rank} of $\A$.
\end{remark}

\begin{remark} \label{rem:rou}
Let $\A$ be a toric arrangement and consider $\chi\in \X(\A)$. Fixing an isomorphism $\Gamma \cong \Z^d$ and corresponding coordinates on $T\cong \C^\times$ we have
$$H_\chi =\{\mathbf{x}=(x_1,\dots,x_d)\in (\C^\times)^d \mid x_1^{c_1}\cdots x_d^{c_d} = 1\}$$ 
where $(c_1,\ldots,c_d)\in \Z^d$ corresponds to $\chi\in\Gamma$. Let $m:=\gcd(c_1,\ldots,c_d)$. Then $H_\chi$ is connected if and only if $\chi$ is primitive, equivalently if $m=1$. In general, the different connected components of $H_{\chi}$ are given by 
\begin{equation}\label{eq:rou}
x_1^{{c_1}/{m}}\cdots x_d^{{c_d}/{m}} = \mu
\end{equation}
where $\mu$ runs over all $m$-th roots of unity.
\end{remark}

\begin{example}\label{ex:A}
Let $\G=\C^\times$ and $\IL=\Z^2$, so $T\cong(\C^\times)^2$.
The columns $\chi_1$, $\chi_2$, and $\chi_3$ of the integer matrix
\[\begin{pmatrix}
2 & -2 & 0 \\
0 & 1 &  1
\end{pmatrix}\]
 define a toric arrangement $\A=\{H_0,H_1,H_2,H_3\}$, where
$H_0$ and $H_1$ denote the two connected components of $H_{\chi_1}$, $H_2:=H_{\chi_2}$, and $H_3:=H_{\chi_3}$.
Considering $\R/\Z\subseteq\C/\Z\cong\C^\times$, the ``real part'' of the arrangement is depicted in \Cref{fig:A}, and the Hasse diagram for the poset of layers is depicted in \Cref{fig:P}.
\begin{figure}[ht]
\begin{subfigure}[t]{.4\textwidth}
\centering
\begin{tikzpicture}
\draw[-] (0,0) -- (2,0) -- (2,2) -- (0,2) -- (0,0);
\draw[ultra thick,blue,-] (0,0) -- (0,2);
\draw[ultra thick,blue,-] (1,0) -- (1,2);
\draw[ultra thick,teal,-] (0,0) -- (2,0);
\draw[ultra thick,purple,-] (0,0) -- (1,2);
\draw[ultra thick,purple,-] (1,0) -- (2,2);
\node at (0,-.2) {};
\end{tikzpicture}
\caption{A (real) toric arrangement $\A$ in $S^1\times S^1$}
\label{fig:A}
\end{subfigure}
\hfill
\begin{subfigure}[t]{.4\textwidth}
\centering
\begin{tikzpicture}[scale=.6]
\node (T) at (0,0) {\scriptsize $T$};
\node[blue] (0) at (-3,2) {\scriptsize $H_0$};
\node[blue] (1) at (3,2) {\scriptsize $H_1$};
\node[purple] (2) at (-1,2) {\scriptsize $H_2$};
\node[teal] (3) at (1,2) {\scriptsize $H_3$};
\node (+) at (-1,4) {\scriptsize $(1,1)$};
\node (-) at (1,4) {\scriptsize $(-1,1)$};
\foreach \x in {0,1,2,3} {
\draw[-] (\x) -- (T) ;
};
\foreach \x in {0,2,3} {
\draw[-] (\x) -- (+) ;
};
\foreach \x in {1,2,3} {
\draw[-] (\x) -- (-) ;
};
\end{tikzpicture}
\caption{The poset of layers $\P(\A)$}
\label{fig:P}
\end{subfigure}
\caption{See \Cref{ex:A,ex:A2,ex:A3}, and \Cref{subsec:rank2circ}.}
\end{figure}
\end{example}

\subsection{Supersolvability}

A subgroup $Y$ of $T$ is \textbf{admissible} if there is a rank-one direct summand $\IL'\subseteq \IL$ such that $Y$ is the image of the injection $\epsilon^*:\Hom(\IL',\G)\to\Hom(\IL,\G)$ induced by the projection $\epsilon:\IL\to\IL'$. 
When $Y$ is admissible, the corresponding projection 
\[p:T\to T/Y \cong\Hom(\IL/\IL',\G)\]
is a section of the map induced by the quotient $q:\IL\to\IL/\IL'$. This allows us to define abelian arrangements
\[
\A_Y:= \{H\in\A \mid H\supseteq Y\}
\qquad 
\A/Y := \{p(H) \mid H\in\A_Y \}\]
in $T$ and $T/Y$, respectively. Note that $\P(\A_Y)$ is by definition a subposet of $\P(\A)$.

The projection $p:T\to T/Y$ restricts to a map on arrangement complements $\pr:M(\A)\to M(\A/Y)$ and induces an isomorphism of posets $\P(\A_Y)\cong\P(\A/Y)$.

\begin{definition}[{\cite[Definitions 2.4.1, 5.1.1]{BD}}]\label{def:MTM}
Let $Y$ be an admissible subgroup of $T$, and $\A$ an abelian arrangement in $T$.
We say $\P(\A_Y)$ is an \textbf{M-ideal} of $\P(\A)$ if for any two distinct $H_1,H_2\in \A\smallsetminus \A_Y$, and any component $X$ of $H_1\cap H_2$, there is some $H_3\in\A_Y$ such that $H_3\supseteq X$.
Say $\P(\A_Y)$ is a \textbf{TM-ideal} if, in addition, the intersection $H\cap Y$ is connected for all $H\in \A\smallsetminus \A_Y$.

Say $\A$ is \textbf{(strictly) supersolvable} if there is a chain
\begin{equation} \label{eq:sschain}
\{\zero\} \subset \P(\A_{Y_1})\subset \P(\A_{Y_2}) \subset\cdots\subset \P(\A_{Y_{d-1}}) \subset\P(\A)
\end{equation}
with each $\P(\A_{Y_r})$ a (T)M-ideal of its successor.
\end{definition}

\begin{remark}
Notice that in \Cref{def:MTM} the rank of $\P(\A_Y)$ is one less than the rank of $\P(\A)$. Henceforth, whenever we say $\P(\A_Y)$ is a corank-one M-ideal of $\P(\A)$, it is assumed that $Y$ is an admissible subgroup of $T$.
\end{remark}

\begin{lemma}\label{lem:layermap}
Let $\A$ be an abelian arrangement in $T\cong\G^d$, and suppose that $\P(\A_Y)$ is a corank-one M-ideal of $\P(\A)$.
Let $p: T\to T/Y$  be the projection to the quotient. Then 
for every $X\in \P(\A)$ we have $p(X)\in \P(\A/Y)$. Moreover, if $X\in \P(\A_Y)$ then $\dim p(X) = \dim (X) -\dim(\G)$, otherwise $\dim p(X) = \dim (X)$.
\end{lemma}
\begin{proof} Let $X\in \P(\A)$. Then $X$ is a coset of a closed connected subgroup of $T$. 

Now choose $ H_1,\ldots,H_k\in \A_Y$, $H_1',\ldots,H_l'\in \A\smallsetminus \A_Y$ such that $X$ is a connected component of $H_1\cap\ldots\cap H_k\cap H_1'\cap\ldots \cap H_l'$, where $k,l\geq 0$. Since $p(H_i')=T/Y$ for all $i=1,\ldots,l$ (e.g., by \cite[Corollary 3.3.2.]{BD}), the coset $p(X)$ is contained in a connected component $W\in \P(\A/Y)$ of $\cap_i p(H_i)$.

If $X\in \P(\A_Y)$ then we may suppose $l=0$, and $H_i= p(H_i)\times Y$ implies $X=p(X)\times Y $. In particular, $\dim (p(X)) = \dim (X) -\dim(\G)$.

If $X\not\in \P(\A_Y)$ then $l>0$, for every $1\leq i<j\leq l$, the definition of M-ideal implies that there is $H\in\A_Y$ such that the connected component of $H_i\cap H_j$ containing $W$ equals the connected component of $H\cap H_i$ containing $W$. Thus we can assume that $X$ is a connected component of an intersection of the form $H_1\cap\ldots \cap H_k \cap H_1'$. Since $H_1'$ is transverse to every $H_i$, 
$\dim(X)=\dim(\cap_i H_i)-\dim(\G)$.
Moreover, since $p(H_1')=T/Y$, $p(X)=W$ and in particular $\dim p(X) = \dim(W) = \dim (\cap_i p(H_i)) = 
\dim(\cap_i H_i)-\dim(\G) = \dim(X).$
\end{proof}

For us, the importance of an M-ideal is that it characterizes when the map $\pr:M(\A)\to M(\A/Y)$ is a fiber bundle \cite[Theorem A]{BD}. Our immediate goal is to show that these bundles are closely related to bundles 
over configuration spaces. We subsequently specialize to toric arrangements where this structure has particularly interesting consequences.

\begin{example}\label{ex:A2}
Recall the toric arrangement from \Cref{ex:A}.
The subgroup $Y=H_0$ yields a TM-ideal $\P(\A_Y) = \{T,H_0,H_1\}$, hence the poset $\P(\A)$ is strictly supersolvable.
The subgroup $Y=H_3$ (or similarly $Y=H_2$) yields an M-ideal $\P(\A_Y)=\{T,H_3\}$ which is not a TM-ideal, since $H_2\cap H_3$ is disconnected.
\end{example}

\subsection{Somewhat ordered configuration spaces}
Given a positive integer $k$ and topological space $X$, denote the \textit{ordered configuration space} by
\[\Conf_k(X) = \{(x_1,\dots,x_k)\in X^k \mid x_i\neq x_j \text{ when } i\neq j \}.\]
The symmetric group $\pg{k}$ acts freely on $\Conf_k(X)\subseteq X^k$ by permuting coordinates. 
The \textit{unordered configuration space} is the quotient space $\Conf_k(X)/\pg{k}$, whose elements are regarded as sets (rather than ordered tuples) of distinct points in $X$. 
More generally, consider a composition of the integer $k$, that is, a sequence $\k=(k_1,\dots,k_m)$ of positive integers satisfying $k=k_1+\cdots+k_m$.
Such a composition 
determines a  subgroup $\pg{\k} := \pg{k_1}\times\dots\times\pg{k_m} \subseteq \pg{k}$.
The \textit{somewhat ordered configuration space} (or \textit{$\k$-ordered configuration space}) is then defined as the quotient
\[\SConf{\k}(X) := \Conf_k(X)/\pg{\k}.\]
An element of $\SConf{\k}(X)$ can be represented by an ordered tuple $(S_1,\dots,S_m)$ of pairwise disjoint subsets of $X$ with $|S_i|=k_i$ for each $i$. 

By a classical result of Fadell and Neuwirth \cite[Theorem 3]{FN} (see also \cite[Theorem 1.1]{FH}), for ordered configuration spaces of a manifold $X$ without boundary, the forgetful map 
\begin{equation} \label{eq:FN}
\Conf_{k+1}(X)\to\Conf_k(X),\quad  (x_1,\dots,x_k,x_{k+1}) \mapsto (x_1,\dots,x_k),
\end{equation}
is a fiber bundle, with fiber homeomorphic to $X$ with $k$ points removed. We refer to this as the \textit{Fadell-Neuwirth bundle}.

\begin{proposition}\label{prop:FN}
Let $X$ be a manifold without boundary, and
let $\k=(k_1,\dots,k_m)$ be a composition of an integer $k$. 
Setting $\kplus = (k_1,\dots,k_m,1)$, the function
$\pi:\SConf{\kplus}(X) \to \SConf{\k}(X)$, given by $(S_1,\dots,S_m,S_{m+1}) \mapsto (S_1,\dots,S_m)$,
is a fiber bundle whose fiber is homeomorphic to $X$ with $k$ points removed.
\end{proposition}
\begin{proof}
The Fadell-Neuwirth bundle \eqref{eq:FN} of ordered configuration spaces is equivariant with respect to the $\pg{\k} = \pg{k_1}\times\dots\times\pg{k_m} \subseteq \pg{k}$ 
actions, hence induces a bundle on the quotients.
\end{proof}

\begin{remark}
Let $\k[n]=(n_1,\dots,n_m)$ be a permutation of the composition $\k=(k_1,\dots,k_m)$, and for $0\leq i\leq m$ let $\k[n_i]=(n_1,\dots,n_{i},1,n_{i+1},\dots,n_m)$.
Then the map $\SConf{\k[n_i]}(X)\to\SConf{\k[n]}(X)$ given by $(S_1,\dots,S_{m+1})\mapsto (S_1,\dots,S_{i-1},S_{i+1},\dots,S_{m+1})$ is a bundle equivalent to the bundle $\pi$ of \Cref{prop:FN}.
\end{remark}

\begin{remark} \label{rem:unorderpull}
The bundle $\SConf{\kplus}(X) \to \SConf{\k}(X)$ of \Cref{prop:FN} may be pulled back from the bundle $\SConf{(k,1)}(X) \to \SConf{(k)}(X)=\Conf_k(X)/\pg{k}$ over the unordered configuration space. 
\end{remark}

\subsection{Monodromy and the Artin representation} \label{subsec:Artin}
If $p\colon E \to B$ is a fiber bundle, with section $s$, choosing basepoints $b_0\in B$ and $e_0\in p^{-1}(b_0)= F\subset E$, the long exact homotopy sequence of the bundle splits, yielding\begin{center}
\begin{tikzcd}
1 \arrow{r} & \pi_1(F,e_0) \arrow{r} & \pi_1(E,e_0) \arrow{r}{p_\sharp} & \pi_1(B,b_0) \arrow{r} \arrow[l,bend left=-33,"s_\sharp",swap] & 1
\end{tikzcd}
\end{center}
Here and below, we use $f_\sharp$ to denote the map on fundamental groups induced by a continuous map $f\colon (X,x_0)\to (Y,f(x_0))$.
Suppressing basepoints, the split exact sequence above realizes the fundamental group $\pi_1(E)$ of the total space as a semidirect product of the fundamental groups of the base and fiber.
The semidirect product structure is determined by an action of $\pi_1(B)$ on $\pi_1(F)$, that is, a homomorphism $\phi\colon \pi_1(B) \to \Aut(\pi_1(F))$. Here, $\Aut(\pi_1(F))$ denotes the group of right automorphisms of $\pi_1(F)$, with group operation $\alpha \cdot \beta$ given by the composition $\beta  \circ \alpha$ for automorphisms $\alpha$ and $\beta$. The homomorphism 
$\phi\colon \pi_1(B) \to \Aut(\pi_1(F))$ is the \textbf{monodromy} of the bundle $p\colon E \to B$.

In the case $X=\C$, the bundle $\SConf{(k,1)}(\C) \to \Conf_k(\C)/\pg{k}$ noted above is equivalent (that is, fiberwise homeomorphic) to the bundle denoted $p_k\colon Y^{k+1}\to B^k$ in 
\cite[\S2]{CSmonodromy}, where $B^k=\Conf_k(\C)/\pg{k}$. As noted there, the monodromy of this bundle is the Artin representation
$\alpha_k\colon B_k \to \Aut(F_k)$, where $B_k$ is the $k$-strand Artin (full) braid group and $\Aut(F_k)$ is the group of right automorphisms of the free group $F_k$. 
See, for instance, \cite{Hansen} as a general reference on braids. 
In terms of the (standard) generators $\sigma_1,\dots,\sigma_{k-1}$ of $B_k$ and $\ys_1^{},\dots,\ys_k^{}$ of $F_k$, the Artin representation is given by
\begin{equation} \label{eq:FullArtin}
\alpha_k(\sigma_i)(\ys_j^{})=\begin{cases}
\ys_i^{}\ys_{i+1}^{}\ys_i^{-1}&\text{if $j=i$,}\\
\ys_i^{}&\text{if $j=i+1$,}\\
\ys_j^{}&\text{otherwise.}
\end{cases}
\end{equation}

Since the Artin representation is faithful, for a braid $\beta$, we often abbreviate the automorphism $\alpha_k(\beta)$ by simply $\beta$.
With this convention, the restriction $\hat{\alpha}_k\colon P_k \to \Aut(F_k)$ of the Artin representation to the pure braid group $P_k<B_k$, with generators 
$a_{i,j}=\sigma_{j-1}^{}\cdots\sigma_{i+1}^{}\sigma_i^2\sigma_{i+1}^{-1}\cdots\sigma_{j-1}^{-1}$, $1\le i<j\le k$, is given by
\begin{equation} \label{eq:PureArtin}
a_{i,j}(\ys_q^{})=\begin{cases}
\ys_i^{}\ys_j^{} \cdot \ys_q^{} \cdot (\ys_i^{}\ys_j^{})^{-1} &\text{if $q=i$ or $q=j$,}\\
[\ys_i^{},\ys_j^{}]\cdot \ys_q^{} \cdot [\ys_i^{},\ys_j^{}]^{-1} &\text{if $i<q<j$,}\\
\ys_q^{}&\text{otherwise,}
\end{cases}
\end{equation}
One can write $\ys_i^{}\ys_j^{} \cdot \ys_q^{} \cdot (\ys_i^{}\ys_j^{})^{-1}={[}\ys_i^{}\ys_j^{},\,\ys_q^{}] \cdot \ys_q^{}$ and $[\ys_i^{},\ys_j^{}]\cdot \ys_q^{} \cdot [\ys_i^{},\ys_j^{}]^{-1}={[}[\ys_i^{},\ys_j^{}],\ys_q^{}]\cdot \ys_q^{}$.

Observe that pure braid automorphisms are IA-automorphisms of the free group $F_k$, inducing the identity on the abelianization.  Also, as noted in \cite[\S2]{CSmonodromy}, the restriction $\hat{\alpha}_k$ of the Artin representation to $P_k=\pi_1(\Conf_k(\C))$ is the monodromy of the Fadell-Neuwirth bundle  $\Conf_{k+1}(\C)\to \Conf_k(\C)$. 

\subsection{Abelian arrangement bundles as pullbacks}\label{sec:pullback}
Let $\G$ be a connected abelian Lie group.
Let $\A$ be an essential abelian arrangement in $T\cong\G^d$ and $Y$ an admissible subgroup of $T$ such that $\P(\A_Y)$ is an M-ideal in $\P(\A)$. 
Then 
the projection $p:T\to T/Y$ restricts to a 
map $\pr:M(\A)\to M(\A/Y)$.
We prove that the restriction $\pr$ 
is a pullback of a configuration space bundle from \Cref{prop:FN}, building on special cases seen in \cite[Theorem 1.1.5]{cohen}, 
\cite[Theorem 3.5.1]{BD}.

\begin{theorem}\label{thm:pullback}
Let $\A$ be an abelian arrangement in $T\cong\G^d$, and suppose that $\P(\A_Y)$ is a corank-one M-ideal of $\P(\A)$.
There is a composition $\k$ and continuous map $g:M(\A/Y) \to \SConf{\k}(\G)$ such that $\pr:M(\A)\to M(\A/Y)$ is the pullback of $\pi:\SConf{\kplus}(\G)\to \SConf{\k}(\G)$ along $g$, as in \Cref{fig:pullback}.
\begin{figure}[h]
\begin{tikzpicture}
\node (A) at (-1,2) {$M(\A)$};
\node (B) at (-1,0) {$M(\A/Y)$};
\node (k) at (2,0) {$\SConf{\k}(\G)$};
\node (1) at (2,2) {$\SConf{\kplus}(\G)$};
\draw[->] (A) -- node[left]{$\pr$} (B) ;
\draw[->] (A) -- node[above]{$h$} (1) ;
\draw[->] (B) -- node[above]{$g$} (k) ;
\draw[->] (1) -- node[right]{$\pi$} (k) ;
\end{tikzpicture}
\caption{Pullback diagram of \Cref{thm:pullback}}
\label{fig:pullback}
\end{figure}
\end{theorem}
\begin{proof}
Write $\A\smallsetminus \A_Y = \{H_1,\dots,H_l\}$.  
We think of $T\cong \G\times(T/Y)$, and for $q=(x,t)\in T$ we let $[q]_1=x$ denote the first coordinate.
By \cite[Corollary 3.3.2]{BD}, for each $i$, the restriction of $p$ to $H_i\subseteq T$ is a covering map $p_i\colon H_i\to T/Y$.
As such, the number $k_i:=|p_i^{-1}(t)|$ is independent of the choice of $t\in T/Y$. 
The sequence $\k=(k_1,\dots,k_l)$ is the composition we will use.

Define the function $g:M(\A/Y)\to\SConf{\k}(\G)$ by
\[g(t) = ([p_1^{-1}(t)]_1,\dots,[p_l^{-1}(t)]_1).\]
This is well-defined since, by \cite[Lemma 3.2.4, Proposition 3.2.5]{BD}, one has $[p_i^{-1}(t)]_1\cap [p_j^{-1}(t)]_1=\emptyset$ for every $t\in M(\A/Y)$ and $i\neq j$.

In order to prove that $g$ is continuous, take an open set $U\subseteq \SConf{\k}(\G)$ and consider $t\in g^{-1}(U)$.
We will construct an open neighborhood of $t$ in $M(\A/Y)$ contained in $g^{-1}(U)$.
Since $\SConf{\k}(\G)$ has the quotient topology from $\Conf_k(\G)$, which has the subspace topology from $\G^k$, we can choose small open sets $U_{ij}\subseteq \G$, for $1\leq i\leq l$ and $1\leq j\leq k_i$, so that
\[U_{11}\times 
\dots\times U_{lk_l} \subseteq\Conf_k(\G)\]
is a representative for an open neighborhood of $g(t)$ in $U$. 
For all $i,j$ the set 
$
(U_{ij}\times T/Y) \cap H_i
$ 
is open in $H_i$. Since covering maps are open, for every $i,j$ the set
$$
V_{ij}:= p_{i}((U_{ij}\times T/Y) \cap H_i) \cap M(\A/Y)
$$
is an open neighborhood of $t$ in $M(\A/Y)$ with 
$g(V_{ij})\subseteq  U$. 
Thus $\bigcap_{ij}V_{ij}$ is the desired open neighborhood of $t$ in $g^{-1}(U)$. 

To complete the diagram \eqref{fig:pullback}, the map $h$ is defined on $M(\A)\subseteq T\cong \G\times(T/Y)$ via $h(t,x)=(g(t),x)$.
The check that this square satisfies the universal property of a pullback is routine (as in the proof of \cite[Theorem 5.3.1]{BD}).
\end{proof}

\begin{remark} \Cref{thm:pullback} implies that the maps $\pr:M(\A)\to M(\A/Y)$ are indeed fiber bundles. This was proved in \cite[Theorem 3.3.1]{BD}, where, for simplicity, the additional technical hypothesis that no two hypersurfaces share a connected component was assumed.
\end{remark}

\begin{remark}\label{rem:orderedpull}
When $\P(\A_Y)$ is a TM-ideal, the composition of \Cref{thm:pullback} is  
$\k=(1,1,\dots,1)$ and the bundle $\pr:M(\A)\to M(\A/Y)$ is a pullback of the Fadell-Neuwirth bundle \eqref{eq:FN} of ordered configuration spaces, recovering \cite[Theorem 5.3.1]{BD}.
\end{remark}

\section{Toric arrangements} \label{sec:toric}

\subsection{Toric arrangement bundles} \label{subsec:tbundles}
In the case that $\G=\C^\times$, there is a close relationship between toric arrangements and configurations of points in the plane.
This in turn has several particularly nice consequences.

\begin{theorem}\label{thm:toricpullback}
Let $\A$ be a toric arrangement, and suppose $\P(\A_Y)$ is a corank-one M-ideal of $\P(\A)$.
\begin{enumerate}
\item \label{item:Cpullback} There is a composition $\k[n]$ and a 
map $f:M(\A/Y)\to \SConf{\k[n]}(\C)$ such that $\pr:M(\A)\to M(\A/Y)$ is the pullback of the bundle $\pi:\SConf{\kplus[n]}(\C)\to\SConf{\k[n]}(\C)$ along $f$.
\item \label{item:coef} There is an integer $n$ and a map $\coefmap\colon M(\A/Y) \to \Conf_n(\C)/\pg{n}$ such that 
$\pr:M(\A)\to M(\A/Y)$ is the pullback of the bundle $\pi\colon \SConf{(n,1)}(\C) \to \SConf{(n)}(\C)= \Conf_n(\C)/\pg{n}$ over the unordered configuration space along $\coefmap$.
\item \label{item:root} If $\P(\A_Y)$ is a TM-ideal,  
there is a map $\rootmap\colon M(\A/Y) \to \Conf_n(\C)$ such that 
$\pr:M(\A)\to M(\A/Y)$ is the pullback of the bundle $\pi\colon \Conf_{n+1}(\C) \to \Conf_n(\C)$ over the ordered configuration space along $\rootmap$.
\end{enumerate}
\end{theorem}
\begin{proof}
From \Cref{thm:pullback}, we have a composition $\k$ and map $g:M(\A/Y)\to\Conf^{\k}(\C^\times)$ through which we can pull back the bundle $\Conf^{\kplus}(\C^\times)\to \Conf^{\k}(\C^\times)$ to the bundle $\pr$, as in the lefthand square of \Cref{fig:toricpullback}. 
\begin{figure}[ht]
\begin{tikzpicture}
\node (A) at (-1,2) {$M(\A)$};
\node (B) at (-1,0) {$M(\A/Y)$};
\node (k) at (2,0) {$\SConf{\k}(\C^\times)$};
\node (1) at (2,2) {$\SConf{\kplus}(\C^\times)$};
\node (C) at (5,0) {$\SConf{\kplus}(\C)$};
\node (C+) at (5,2) {$\SConf{(\k,1,1)}(\C)$};
\node (D) at (8,0) {$\SConf{(n)}(\C)$};
\node (D+) at (8,2) {$\SConf{(n,1)}(\C)$};
\draw[->] (A) -- node[left]{$\pr$} (B) ;
\draw[->] (A) -- (1) ;
\draw[->] (B) -- node[above]{$g$} (k) ;
\draw[->] (1) -- (k) ;
\draw[->] (k) -- node[above]{$z$} (C) ;
\draw[->] (C) -- node[above]{$w$} (D) ;
\draw[->] (1) -- (C+) ;
\draw[->] (C+) -- (D+) ;
\draw[->] (C+) -- (C) ;
\draw[->] (D+) -- (D) ;
\end{tikzpicture}
\caption{Pullback diagram of \Cref{thm:toricpullback}}
\label{fig:toricpullback}
\end{figure}
We further have a continuous map $z\colon\SConf{\k}(\C^\times)\to \SConf{\kplus}(\C)$, given by $(S_1,\dots,S_m)\mapsto (S_1,\dots,S_m,0)$, making the middle square of \Cref{fig:toricpullback} a pullback diagram. Letting $n=k+1$, from \Cref{rem:unorderpull}, we also have a map $w \colon \SConf{\kplus}(\C) \to  \SConf{(n)}(\C) = \Conf_n(\C)/\pg{n}$ making the righthand square a pullback.

Parts \eqref{item:Cpullback} and \eqref{item:coef} of the theorem follow with $\k[n]=\kplus$, $f=z\circ g$, and $\coefmap=w\circ z \circ g$.

For part \eqref{item:root}, as noted in \Cref{rem:orderedpull}, if $\P(\A_Y)$ is a TM-ideal, the composition of \Cref{thm:pullback} is the trivial composition $\k=(1,1,\dots,1)$. Consequently, 
$\k[n]=\kplus$ is trivial as well, and $\SConf{\kplus}(\C)$ is the ordered configuration space $\Conf_n(\C)$. Setting $\rootmap=z \circ g$ in this instance completes the proof.
\end{proof}

\begin{corollary}\label{cor:section}
Let $\A$ be a toric arrangement, and suppose $\P(\A_Y)$ is a corank-one M-ideal of $\P(\A)$.
Then the associated fiber bundle $\pr:M(\A)\to M(\A/Y)$ admits a section.
\end{corollary}
\begin{proof}
By \Cref{thm:toricpullback} and \Cref{rem:unorderpull}, we need only check that the bundle $\pi:\SConf{(k,1)}(\C)\to\Conf_k(\C)/\pg{k}$ has a section.
A section of $\pi$ is obtained by mapping a set $S=\{x_1,\dots,x_k\}$ of $k$ distinct points in $\C$ to the configuration $(S,\max\{|x_1|,\dots,|x_k|\}+1)$ in $\SConf{(k,1)}(\C)$. 
\end{proof}

\begin{remark}
The existence of a section can be extended to supersolvable abelian arrangements when $\G$ is noncompact, in a similar fashion to \Cref{cor:section}. 
When  $\G$ is compact, if there is an $i$ with $k_i=1$, then the bundle $\Conf^{\kplus}(\G)\to\Conf^{\k}(\G)$ has a section, since
for a configuration $(S_1,\dots,S_m)$,  we can add a point near the (unique) point in $S_i$.
This section can then be pulled back to a section {of} $M(\A)\to M(\A/Y)$ as long as there is some $H\in\A\smallsetminus\A_Y$ such that $H\cap Y$ is connected.
\end{remark}

\subsection{Polynomials} \label{subsec:polys}

Parts \eqref{item:coef} and \eqref{item:root} of \Cref{thm:toricpullback} bring to the fore the relationship between (strictly) supersolvable toric arrangements and Hansen's theory of polynomial coverings \cite{Hansen} and the associated braid bundles of \cite{CSmonodromy}. In these situations, choices of the pullback maps $\coefmap$ and $\rootmap$ may be obtained directly from the characters defining the toric arrangement.

\begin{remark}\label{confdisc}
The unordered configuration space $\Conf_n(\C)/\pg{n}$ may be realized as the complement of the discriminant in $\C^n$, the space of monic complex polynomials of degree $n$ with distinct roots. With this identification, the covering map $\Conf_n(\C)\to \Conf_n(\C)/\pg{n}$ takes an $n$-tuple $(x_1,\dots,x_n)$ of distinct complex numbers to the polynomial (in $z$) with these roots, namely $\prod_{i=1}^n (z-x_i)$.
\end{remark}

Now let $\A$ be an essential supersolvable toric arrangement in $(\C^\times)^{d+1}$, with $\P(\A_Y)$ a corank-one M-ideal of $\P(\A)$. Write $\A\smallsetminus\A_Y=\{H_1,\dots,H_l\}$. Choosing coordinates $(x_1,\dots,x_d,y)=(\mathbf{x},y)$ appropriately, for every $j=1,\ldots l$ the hypersurface $H_j$ is defined by 
\[
\begin{aligned}
H_j&=\{(x_1,\dots,x_d,y) \in (\C^\times)^{d+1} \mid y^{m_{j,0}} - \mu_j x_1^{m_{j,1}}x_2^{m_{j,2}} \cdots x_d^{m_{j,d}}=0\}\\
&=\{(\mathbf{x},y) \in (\C^\times)^{d+1} \mid y^{m_{j,0}} - \mu_j\mathbf{x}^{\mathbf{m}_j}=0\},
\end{aligned}
\]
where $m_{j,0}\in \Z_{>0}$ is a positive integer, $\mathbf{m}_j=(m_{j,1},\dots,m_{j,d}) \in \Z^d$, and $\mu_j$ is a root of unity (cf.~\Cref{rem:rou}). Since $\A$ is supersolvable over $\B=\A/Y$, the map $f\colon M(\B) \times \C \to \C$ given by
\[
f(\mathbf{x},y) = y \prod_{j=1}^l(y^{m_{j,0}} - \mu_j\mathbf{x}^{\mathbf{m}_j}) = y^n + \sum_{i=1}^n a_i(\mathbf{x}) y^{n-i}
\]
is a simple Weierstrass polynomial on $M(\B)$ in the sense of \cite{Hansen}: the coefficient maps $a_i\colon M(\B) \to \C$ are continuous, and, for each $\mathbf{x}\in M(\B)$, the polynomial $f(\mathbf{x},y) \in \C[y]$ has distinct roots. 
Note that the coefficient map $a_n$ is the constant function $a_n(\mathbf{x})=0$. 
Sending $\mathbf{x}$ to the set of roots of the polynomial $f(\mathbf{x},y) = y^n + \sum_{i=1}^n a_i(\mathbf{x}) y^{n-i}$ defines a map  
$\coefmap \colon M(\B) \to \Conf_n(\C)/\pg{n}$. Identifying the unordered configuration space with the complement of the discriminant in $\C^n$ via \Cref{confdisc}, this corresponds to sending $\mathbf{x}$ to (the coefficients of) the polynomial $f(\mathbf{x},y)$. Following Hansen \cite{Hansen}, we call $\coefmap \colon M(\B) \to \Conf_n(\C)/\pg{n}$ the \textbf{coefficient map}.

If, moreover, $\P(\A_Y)$ is a TM-ideal, then factoring and reindexing as needed, we can assume that $m_{j,0}=1$ for each $j$. In this instance, the simple Weierstrass polynomial $f$ is completely solvable, meaning it factors as
\[
f(\mathbf{x},y) = y \prod_{j=1}^l(y - \mu_j\mathbf{x}^{\mathbf{m}_j}) = \prod_{i=1}^n (y-b_i(\mathbf{x})),
\]
where $\mu_j$ is some root of unity, with continuous root maps $b_i\colon M(\B) \to \C$. Since the roots are distinct, this defines a \textbf{root map} $\rootmap \colon M(\B) \to \Conf_n(\C)$, given by 
$\mathbf{x} \mapsto (b_1(\mathbf{x}),\dots,b_n(\mathbf{x}))$. Note that $n=l+1$ in this instance.

\begin{remark}\label{rem:twoabs} The maps $\k[a]$ and $\k[b]$ defined here via the polynomial $f$ are instances of the corresponding maps of \Cref{thm:toricpullback}.
\end{remark}

These considerations yield the following versions of parts \eqref{item:coef} and \eqref{item:root} of \Cref{thm:toricpullback}, which may be checked directly. 
Recall from \S\ref{subsec:Artin} that the monodromy of a bundle $p\colon E \to B$, with fiber $F$ and a (fixed) section, is the homomorphism from $\pi_1(B)$ to $\Aut(\pi_1(F))$, the group of (right) automorphisms of $\pi_1(F)$, giving the action of the fundamental group of the base on that of the fiber. 
Also recall the Artin representation discussed in \S\ref{subsec:Artin}, and that the homomophism on fundamental groups induced by a map $f$ is denoted by $f_\sharp$. 
\begin{proposition} \label{prop:coef/root}
Let $\A$ be a toric arrangement, and suppose $\P(\A_Y)$ is a corank-one M-ideal of $\P(\A)$. 
Let $f$ be the associated Weierstress polynomial with $\B=\A/Y$.
\begin{enumerate}
\item \label{item:coef1} The bundle  
$\pr:M(\A)\to M(\B)$ is the pullback of $\pi\colon \SConf{(n,1)}(\C) \to \Conf_n(\C)/\pg{n}$ along the coefficient map 
$\coefmap \colon M(\B) \to \Conf_n(\C)/\pg{n}$, given by $\mathbf{x} \mapsto y \prod_{j=1}^l(y^{m_{j,0}} - \mathbf{x}^{\mathbf{m}_j})$.
The monodromy of the bundle $\pr:M(\A)\to M(\B)$ factors as $\alpha_n \circ \coefmap_\sharp$, where $\alpha_n\colon B_n \to \Aut(F_n)$ is the Artin representation.
\item \label{item:root1} If $\P(\A_Y)$ is a TM-ideal, the bundle  
$\pr:M(\A)\to M(\B)$ is the pullback of $\pi\colon \Conf_{n+1}(\C) \to \Conf_n(\C)$ along the root map $\rootmap \colon M(\B) \to \Conf_n(\C)$, given by 
$\mathbf{x} \mapsto (0,\mu_1\mathbf{x}^{\mathbf{m}_1},\dots,\mu_l\mathbf{x}^{\mathbf{m}_l})$. The monodromy of the bundle $\pr:M(\A)\to M(\B)$ factors as $\hat{\alpha}_n  \circ \rootmap_\sharp$, where $\hat{\alpha}_n\colon P_n \to \Aut(F_n)$ is the restriction of the Artin representation.
\end{enumerate}
\end{proposition}

\begin{example}\label{ex:A3}

Recall the toric arrangement $\A$ from \Cref{ex:A}, which has by \Cref{ex:A2} a TM-ideal $\P(\A_{H_0})$ and an M-ideal $\P(\A_{H_3})$.

The quotient by $H_0$ then induces a fiber bundle $\pr\colon M(\A)\to M(\A/H_0)=\C\smallsetminus\{0,-1,1\}$, which can be pulled back from an ordered, or unordered, configuration space bundle:
\begin{center}
\begin{tikzpicture}
\node (M) at (0,1) {$\C\smallsetminus\{0,-1,1\}$};
\node (C*) at (3,1) {$\Conf_2(\C^\times)$};
\node (C) at (6,1) {$\Conf_3(\C)$};
\node (U) at (9,1) {$\Conf_3(\C)/\pg{3}$};
\draw[->] (M) -- node[above]{$g$} (C*);
\draw[->] (C*) -- node[above]{$z$} (C);
\draw[->] (C) -- node[above]{$w$} (U);
\node (x) at (0,0) {$x$};
\node (g) at (3,0) {$(x^2,1)$};
\node (z) at (6,0) {$(0,x^2,1)$};
\node (u) at (9,0) {$\{0,x^2,1\}$};
\draw[|->] (x)--(g);
\draw[|->] (g)--(z);
\draw[|->] (z)--(u);
\end{tikzpicture}
\end{center}
In particular, the root map $\rootmap\colon \C\smallsetminus\{0,-1,1\} \to \Conf_3(\C)$ is given by $\rootmap(x)=(0,x^2,1)$. The induced homomorphism 
$\rootmap_\sharp\colon F_3 \to P_3$ may be obtained from the calculations of \S\ref{subsec:rank2circ} below. If $F_3=\langle \xs_0,\xs_1,\xs_2\rangle$, where $\xs_0,\xs_1,\xs_2$ are represented by based loops $\gamma_0$, $\gamma_1$, and $\gamma_2$ about $0$, $-1$, and $1$ respectively (see \S\ref{subsec:r2pi1} for explicit formulas), we have 
\begin{equation}\label{bsharpex}
\rootmap_\sharp(\xs_0)=a_{1,2}^2,\quad \rootmap_\sharp(\xs_1)=a_{1,2}^{}a_{2,3}^{}a_{1,2}^{-1},\quad \rootmap_\sharp(\xs_2)=a_{2,3}^{}.
\end{equation}
The quotient by $H_3$ induces a fiber bundle $\pr\colon M(\A)\to M(\A/H_3)=\C\smallsetminus\{0,1\}$, which can be pulled back from a configuration space bundle through $g$, $z\circ g$, or $w\circ z\circ g$, as below:
\begin{center}
\begin{tikzpicture}
\node (M) at (0,1) {$\C\smallsetminus\{0,1\}$};
\node (C*) at (4,1) {$\SConf{(2,1,1)}(\C^\times)$};
\node (C) at (8,1) {$\SConf{(2,1,1,1)}(\C)$};
\node (U) at (12,1) {$\Conf_5(\C)/\pg{5}$};
\draw[->] (M) -- node[above]{$g$} (C*);
\draw[->] (C*) -- node[above]{$z$} (C);
\draw[->] (C) -- node[above]{$w$} (U);
\node (x) at (0,0) {$y$};
\node (g) at (4,0) {$(\{\pm \sqrt{y}\},-1,1)$};
\node (z) at (8,0) {$(\{\pm \sqrt{y}\},-1,1,0)$};
\node (u) at (12,0) {$\{-\sqrt{y},\sqrt{y},-1,1,0\}$};
\draw[|->] (x)--(g);
\draw[|->] (g)--(z);
\draw[|->] (z)--(u);
\end{tikzpicture}
\end{center}
The coefficient map $\coefmap \colon \C\smallsetminus\{0,1\} \to \Conf_5(\C)/\pg{5}$ is given by (the roots of) $\coefmap(y)=x(x^2-1)(x^2-y) \in \C[x]$. 
Let $\pi_1(\C\smallsetminus\{0,1\})=F_2=\langle \us_0,\us_1\rangle$, where $\us_j$ is represented by a counterclockwise circular path $\mu_j$ of radius $1/2$ based at $1/2$ and centered at $j$. Lifts of the loops $\coefmap \circ \mu_j$ in $\Conf_5(\C)$ are depicted in Figure \ref{fig:BraidReps}. These paths may be used to check that the homomorphism $\coefmap_\sharp\colon F_2 \to B_5$ induced by $\coefmap$ is given by $\coefmap_\sharp(\us_0)=\sigma_2\sigma_3\sigma_2$ and $\coefmap_\sharp(\us_1)=\sigma_1^2\sigma_4^2$.
\begin{figure}[h] 
\includegraphics[scale=.4]{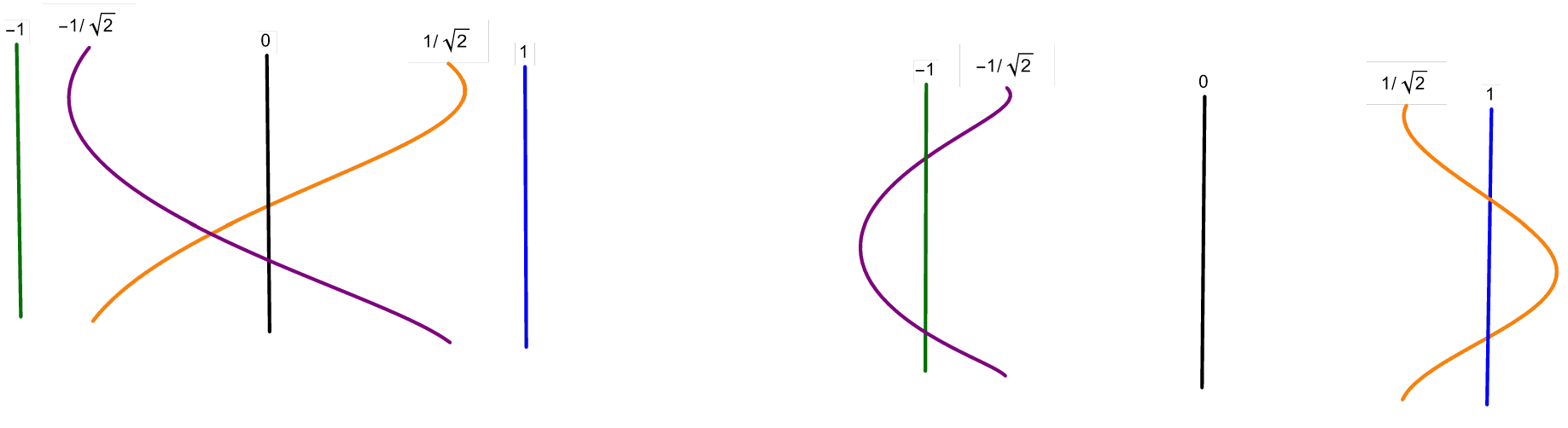}
\caption{Paths in $\Conf_5(\C)$ representing $\coefmap_\sharp(\us_0)$ and $\coefmap_\sharp(\us_1)$.
} \label{fig:BraidReps} 
\end{figure}
\end{example}

\subsection{Fundamental Group} \label{subsec:group}
For a supersolvable toric arrangement $\A$, the results of \S\S\ref{subsec:tbundles}--\ref{subsec:polys} may easily be used to see that the complement $M(\A)$ is a $K(G,1)$-space, where $G=G(\A)=\pi_1(M(\A))$, as shown in \cite[Corollary B]{BD}. These results have further topological and group theoretic implications. We begin with several properties of the group $G$.

An iterated semidirect product of finitely generated free groups $G=F_{n_r} \rtimes\bigl(F_{n_{r-1}} (\rtimes \dots \rtimes (F_{n_2} \rtimes F_{n_1}\bigr)$ is said to be an \textbf{almost-direct product} if the action of the group $\rtimes_{i=1}^{j} F_{n_i}$ on $H_1(F_{n_k};\Z)$ is trivial
for each $j$ and $k$ with $1\le j<k \le r$. That is, each of the homomorphisms $\rtimes_{i=1}^{k-1} F_{n_i} \to \Aut(F_{n_k})$ determining the iterated semidirect product structure of $G$ has image contained in the subgroup of IA-automorphisms, 
consisting of automorphisms which induce the identity on the abelianization of $F_{n_k}$.

\begin{corollary} \label{cor:almostdirect}
If $\A$ is a supersolvable toric arrangement, then the fundamental group of the complement $\pi_1(M(\A))$ is an iterated semidirect product of free groups, the constituent free groups acting on one another by braid automorphisms.
If $\A$ is a strictly supersolvable toric arrangement, then $\pi_1(M(\A))$ is an almost-direct product of free groups, the constituent free groups acting on one another by pure braid automorphisms.
\end{corollary}
\begin{proof}
We proceed by induction on the rank of $\A$. As the base case is clear, assume that the rank of $\A$ is greater than one. 
Since $\A$ is supersolvable, we have by \Cref{prop:coef/root} 
a supersolvable arrangement $\B$ and fiber bundle $M(\A)\to M(\B)$ whose fiber $F$ is homeomorphic to $\C$ with finitely many points removed. 
The associated long exact sequence on homotopy groups reduces to a short exact sequence on the fundamental groups
\[ 1 \longrightarrow \pi_1(F) \longrightarrow \pi_1(M(\A)) \longrightarrow \pi_1(M(\B)) \longrightarrow 1 \]
This short exact sequence splits by \Cref{cor:section}, implying that $\pi_1(M(\A))$ is a semidirect product of $\pi_1(M(\B))$ (an iterated semidirect product of free groups with actions given by braid automorphisms, by induction) and $\pi_1(F)$ (a free group). By  \Cref{prop:coef/root}\,\eqref{item:coef1}, the monodromy of the bundle $M(\A)\to M(\B)$ factors through a braid group, so $\pi_1(M(\B))$ acts on $\pi_1(F)$ by braid automorphisms.

If we moreover have that $\A$ is strictly supersolvable, then we can choose $\B$ to also be strictly supersolvable and the fiber bundle $M(\A) \to M(\B)$ is pulled back from the ordered configuration space bundle, 
via \Cref{thm:toricpullback}\,\eqref{item:root}.  By induction,  $\pi_1(M(\B))$ is an almost-direct product of free groups with actions given by pure braid automorphisms. By  \Cref{prop:coef/root}\,\eqref{item:root1}, the monodromy of the bundle $M(\A)\to M(\B)$ factors through a pure braid group, so $\pi_1(M(\B))$ acts on $\pi_1(F)$ by pure braid automorphisms, 
which as noted in \S\ref{subsec:Artin} act trivially on homology. 
\end{proof}

Recall that a discrete group is said to be linear if it admits a faithful, finite-dimensional linear representation (over some field). From work of Bigelow \cite{Big} and Krammer \cite{Kra}, it is known that the Artin braid group is linear. This, together with the above, can be used to establish the linearity of supersolvable toric arrangement groups. The proof given below follows that of \cite{CCP}, where supersolvable hyperplane arrangement groups were shown to be linear.
\begin{corollary}
If $\A$ is a supersolvable toric arrangement, then the fundamental group of the complement $\pi_1(M(\A))$ is a linear group.
\end{corollary}
\begin{proof}
We again proceed by induction on the rank of $\A$. The base case is clear as the fundamental group of the complement of a rank one toric arrangement is a finitely generated free group.

For $\A$ supersolvable, as above, we have a corank one supersolvable arrangement $\B$ and fiber bundle $M(\A)\to M(\B)$. By \Cref{thm:toricpullback}\,\eqref{item:coef}, this bundle may be realized as a pullback of the bundle $\Conf^{(n,1)}(\C) \to \Conf_n(\C)/\pg{n}$ over the unordered configuration space, with fiber $\C\smallsetminus\{n\ \text{points}\}$. This yields a commutative diagram of fundamental groups with split short exact rows
\[\begin{tikzcd}
	1 && F_n && \pi_1(M(\A)) && \pi_1(M(\B)) && 1 \\
	1 && F_n && \pi_1(\Conf^{(n,1)}(\C)) && \pi_1(\Conf_n(\C)/\pg{n}) && 1
	\arrow[from=1-3, to=1-5] 
	\arrow[from=1-5, to=1-7]
	\arrow[from=1-1, to=1-3]
	\arrow[from=1-7, to=1-9]
	\arrow[from=2-1, to=2-3]
	\arrow[from=2-3, to=2-5]
	\arrow[from=2-5, to=2-7]
	\arrow[from=2-7, to=2-9]
	\arrow[from=1-3, to=2-3]
	\arrow[from=1-5, to=2-5]
	\arrow[from=1-7, to=2-7]
 \arrow["{=}", from=1-3, to=2-3]
\end{tikzcd}\]
realizing the group $\pi_1(M(\A))$ as a pullback. 

The fundamental group $\pi_1(\Conf_n(\C)/\pg{n})$ is the $n$-strand Artin braid group $B_n$, which as noted above is linear. The group 
$\pi_1(\Conf^{(n,1)}(\C))$ may be realized as the subgroup of the $(n+1)$-strand braid group $B_{n+1}$ 
for which the endpoint of the last strand is fixed, 
so is also linear. Assuming inductively that $\pi_1(M(\B))$ is linear, it follows that 
the pullback $\pi_1(M(\A))$ is also linear, as it is a subgroup of the product $\pi_1(\Conf^{(n,1)}(\C))\times \pi_1(M(\B))$ of linear groups.
\end{proof}

\begin{remark} \label{rem:presentations}
If $\A=\A_r$ is supersolvable of rank $r$, then from \eqref{eq:sschain} we have an increasing chain of M-ideals $\P(\A_{j})$, $1\le j <r$, corresponding supersolvable arrangements $\A_j:= \A_{Y_j}$, so that $\A_j=\A_{j+1}/Y_j$, 
and coefficient maps $\coefmap_j \colon M(\A_{j-1}) \to \Conf_{n_{j}}(\C)/\pg{n_j}$ for $j\geq 2$. 
Setting $n_1=1+|\A_1|$, 
 the fundamental group $G=G(\A)=\pi_1(M(\A))$ is an iterated semidirect product of free groups $G=\rtimes_{j=1}^{r} F_{n_j}$, acting upon one another by braid automorphisms by \Cref{prop:coef/root}\,\eqref{item:coef1}.

For each $j$, $2\le j \le r$, let $\phi_j={\alpha}_{n_j} \circ (\coefmap_j)_\sharp
\colon \pi_1(M(\A_{j-1})) \to \Aut(F_{n_j})$. If $F_{n_j}=\langle \ys_{p,j}^{}\ 1\le p \le n_j\rangle$, the group $G$ has generators 
$\ys_{p,j}$, $1\le p \le n_j$, $1\le j \le r$, and relations 
\begin{equation} \label{eq:pi1rels}
\ys_{p,i}^{-1} \ys_{q,j}^{} \ys_{p,i}^{}=\phi_j(\ys_{p,i}^{})(\ys_{q,j}^{}),\ 1\le p \le n_i, 1\le q \le n_j, 1\le i < j \le r.
\end{equation}

If, further, $\A$ is strictly supersolvable, we have corresponding root maps $\rootmap_j \colon M(\A_{j-1}) \to \Conf_{n_{j}}(\C)$ and the homomorphism
$\phi_j\colon \pi_1(M(\A_{j-1})) \to \Aut(F_{n_j})$ may be expressed as $\phi_j=\hat{\alpha}_{n_j} \circ (\rootmap_j)_\sharp$. In this instance, $G$ is an almost-direct product of free groups, acting upon one another by pure braid automorphisms by \Cref{prop:coef/root}\,\eqref{item:root1}.

Since $\phi_j$ is the composition of the Artin representation and the homomorphism induced by the coefficient or root map, determining the latter yields an explicit presentation of the group $G=\pi_1(M(\A))$. We illustrate this with our running example next. See \S\ref{subsec:r2pi1} and \S\ref{subsec:typeC2} for further illustrations.
\end{remark}

\begin{example}\label{ex:A4}
Recall the toric arrangement $\A$ from \Cref{ex:A}, and the associated fiber bundles from \Cref{ex:A3}.

For the first of these bundles, $\pr\colon M(\A)\to M(\A/H_0)=\C\smallsetminus\{0,-1,1\}$, the action of the fundamental group of the base $F_3=\langle \xs_0,\xs_1,\xs_2\rangle$ on that of the fiber  $F_3=\langle \ys_1,\ys_2,\ys_3\rangle$ is the composition $\phi=\hat\alpha_3 \circ \rootmap_\sharp$ of the root map induced homomorphism $\rootmap_\sharp\colon F_3 \to P_3$ and the Artin representation $\hat\alpha_3\colon P_3 \to  \Aut(F_3)$. Computing with the expression of $\rootmap_\sharp$ from \eqref{bsharpex} in \Cref{ex:A3} and with the Artin representation given in  
\eqref{eq:PureArtin} yields a presentation of $\pi_1(M(\A))$ with generators  $\xs_0,\xs_1,\xs_2,\ys_1,\ys_2,\ys_3$ and relations  
\[u^{-1}v u = \phi(u)(v)=w(u,v) \cdot v\cdot w(u,v)^{-1}\] for $u\in\{\xs_0,\xs_1,\xs_2\}$, $v\in\{\ys_1,\ys_2,\ys_3\}$, and
\[
\begin{matrix}
\begin{aligned}
    w(\xs_0^{},\ys_1^{})&=(\ys_1^{}\ys_2^{})^{2}, \\
    w(\xs_0^{},\ys_2^{})&= \ys_1^{}\ys_2^{}\ys_1^{},\\
    w(\xs_0^{},\ys_3^{})&=1,
\end{aligned}  
&
\begin{aligned}
    w(\xs_1^{},\ys_1^{})&=[\ys_1^{}\ys_2^{}\ys_3^{} \ys_2^{-1} \ys_1^{-1},\,\ys_2], \\  
    w(\xs_1^{},\ys_2^{})&= \ys_1^{}\ys_2^{}\ys_3^{} \ys_2^{-1} \ys_1^{-1},\\
    w(\xs_1^{},\ys_3^{})&=[\ys_2^{-1},\ys_1^{-1}]\cdot \ys_2,
\end{aligned}  
&
\begin{aligned}
    w(\xs_2^{},\ys_1^{})&=1, \\
    w(\xs_2^{},\ys_2^{})&= \ys_2^{}\ys_3^{},\\
    w(\xs_2^{},\ys_3^{})&=\ys_2^{}.
\end{aligned}  
\end{matrix}\]

For the second bundle, $\pr\colon M(\A)\to M(\A/H_3)=\C\smallsetminus\{0,1\}$, the action of the fundamental group of the base $F_2=\langle \us_0,\us_1\rangle$ on that of the fiber  $F_5=\langle \vs_1,\dots,\vs_5\rangle$ is the composition $\psi=\alpha_5 \circ \coefmap_\sharp$ of  the coefficient map induced homomorphism $\coefmap_\sharp\colon F_2 \to B_5$ and the Artin representation $\alpha_5\colon B_5 \to  \Aut(F_5)$. Computing with the expression of $\coefmap_\sharp$ given in \Cref{ex:A3} and \eqref{eq:FullArtin} yields a presentation of $\pi_1(M(\A))$ with generators  $\us_0,\us_1,\vs_1,\vs_2,\vs_3,\vs_4,\vs_5$ and relations  
\[\begin{matrix}
\begin{aligned}
\us_0^{-1}\vs_1^{}\us_0^{}&=\vs_1^{},\\
\us_0^{-1}\vs_2^{}\us_0^{}&= \vs_2^{}\vs_3^{}\vs_4^{}\vs_3^{-1}\vs_2^{-1},\\
\us_0^{-1}\vs_3^{}\us_0^{}&=\vs_2^{}\vs_3^{}\vs_2^{-1},
\end{aligned}  
&
\begin{aligned}
\us_0^{-1}\vs_4^{}\us_0^{}&=\vs_2^{},\\
\us_0^{-1}\vs_5^{}\us_0^{}&=\vs_5,\\
\ &
\end{aligned}  
&
\begin{aligned}
\us_1^{-1}\vs_1^{}\us_1^{}&=\vs_1^{}\vs_2^{}\vs_1\vs_2^{-1}\vs_1^{-1}, \\
\us_1^{-1}\vs_2^{}\us_1^{}&= \vs_1^{}\vs_2^{}\vs_1^{-1},\\
\us_1^{-1}\vs_3^{}\us_1^{}&=\vs_3^{},
\end{aligned}  
&
\begin{aligned}
\us_1^{-1}\vs_4^{}\us_1^{}&=\vs_4^{}\vs_5^{}\vs_4^{}\vs_5^{-1}\vs_4^{-1},\\
\us_1^{-1}\vs_5^{}\us_1^{}&=\vs_4^{}\vs_5^{}\vs_4^{-1}.\\
\ &
\end{aligned}  
\end{matrix}\]

Rewriting these relations using $\vs_2^{}=\us_0^{-1}\vs_4^{}\us_0^{}$ yields a presentation with $6$ generators and $9$ relations. 
Letting $\ws=\ys_1^{}\ys_2^{}\ys_3^{}\ys_2^{-1}\ys_1^{-1}$, one
can check that the correspondence between this presentation and that arising from the first bundle is given by
\[
\vs_1^{} \mapsto \xs_1  [\ws,\,\ys_2], \quad
\vs_3^{} \mapsto \xs_0^{}, \quad \vs_5^{} \mapsto \xs_2^{}, \quad \us_0^{} \mapsto \ys_1, \quad \vs_4 \mapsto \ys_1^{}\ys_2^{}\ys_1^{-1}, \quad 
\us_1 \mapsto \ws.
\]
\end{example}

\section{Strictly supersolvable toric arrangements} \label{sec:ssta}
Let $\A=\A_r$ be a strictly supersolvable toric arrangement of rank $r$, with corresponding TM-ideals $\P(\A_{Y_j})$ and arrangements $\A_j=\A_{j+1}/Y_j$, $1\le j < r$. 
From \Cref{prop:coef/root}\,\eqref{item:root1}, we have associated
root maps $\rootmap_j \colon M(\A_{j-1}) \to \Conf_{n_{j}}(\C)$, where $n_j=1+|\A_{j}|-|\A_{j-1}|$.  
In this section, we focus on 
several invariants which may be obtained from
the sequence
$\bigl((\rootmap_2)_*,\dots,(\rootmap_r)_*\bigr)$ of \textbf{homological root homomorphisms}, where
\begin{equation} \label{eq:homroothom}
 (\rootmap_j)_* \colon H_1(M(\A_{j-1});\Z) \longrightarrow H_1(\Conf_{n_j}(\C);\Z).   
\end{equation}

\subsection{Lower central series Lie algebra}
We first investigate the (integral) lower central series (LCS) Lie algebra of the fundamental group $G(\A)=\pi_1(M(\A))$ of the complement of a strictly supersolvable toric arrangement $\A$. We denote this Lie algebra by $\lie(G(\A))$, or more briefly $\lie(\A)$.  We begin with a brief discussion of the LCS Lie algebra of the pure braid group, sometimes referred to as the universal Yang-Baxter Lie algebra, which will play a prominent role in what follows. 

\begin{example} \label{ex:lcsPn}  
The structure of the LCS Lie algebra of the pure braid group $P_n=\pi_1(\Conf_n(\C))$ was determined by Kohno \cite{Kohno}. The Lie algebra $\lie(P_n)$ is generated by $A_{i,j}=[a_{i,j}]$, $1\le i<j \le n$,  the homology classes of the generator $a_{i,j}$ of $P_n$, and has relations
\[
[A_{i,j},A_{k,l}]=0 \ \text{for}\ i,j,k,l\ \text{distinct, and}\ [A_{q,k},A_{i,j}+A_{i,k}+A_{j,k}]=0\ \text{for}\ q=i,j. 
\]
From this description, it follows that $\lie(P_{n+1})$ is the semidirect product of $\lie(P_n)$ by $\LL[n]$, the free Lie algebra generated by $A_{i,n+1}$, $1\le i\le n$, determined by the Lie homomorphism $\theta_n\colon \lie(P_n) \to \Der(\LL[n])$ given by $\theta_n(A_{i,j})=\ad(A_{i,j})$.
From the relations above, the adjoint action of $\lie(P_n)$ on $\LL[n]$ is given by
\begin{equation} \label{eq:infbraid}
\begin{aligned}
\theta_n(A_{i,j})(A_{q,n+1})&=\ad(A_{i,j})(A_{q,n+1})\\
&=[A_{i,j},A_{q,n+1}]=\begin{cases}
    [A_{q,n+1},A_{i,n+1}+A_{j,n+1}]&\text{if $q=i,j$,}\\
    0&\text{otherwise.}
\end{cases}
\end{aligned}
\end{equation}
We will refer to \eqref{eq:infbraid}, resp., the underlying relations, as the \textbf{infinitesimal pure braid relations}.
\end{example}
\begin{theorem} \label{thm:LCSliealg}
For $\A$ strictly supersolvable, the lower central series Lie algebra $\lie(\A)$ of the fundamental group $G(\A)=\pi_1(M(\A))$ 
is an iterated semidirect product of free Lie algebras, determined by the sequence of homological root homomorphisms 
and the infinitesimal pure braid relations. 
\end{theorem}
\begin{proof}
For $\A$ strictly supersolvable of rank $r$, the fundamental group $G(\A)=\rtimes_{i=1}^{r} F_{n_i}$ is an almost-direct product of free groups. The LCS Lie algebra of the free group $F_n$ is the free Lie algebra $\LL[n]$. From the almost-direct product structure of $G(\A)$, we have an isomorphism of abelian groups $\lie(\A) \cong \LL[n_1]\oplus\dots\oplus \LL[n_r]$ as in \cite[Theorem 3.1]{FR} (see also \cite[Theorem 4.4]{CCX}). 

We must show that the sequence of homological root homomorphisms, together with \eqref{eq:infbraid}, determine the 
iterated semidirect product structure of the Lie algebra $\lie(\A)$. 
This is accomplished by induction on the rank $r$ of $\A$. 
In the base case $r=1$, there is nothing to prove as the fundamental group $G(\A)$ is a finitely generated free group, the lower central series Lie algebra $\lie(\A)$ is a free Lie algebra, and the sequence of homological root homomorphisms is vacuous.

For the general case, write $\B=\A_{r-1}$, $n=n_r$, and denote the root map $\rootmap_r$ by simply $\rootmap$. By induction, the LCS Lie algebra $\lie(\B)$ of $G(\B)$ is an iterated semidirect product of free Lie algebras determined by \eqref{eq:infbraid} and the (truncated) sequence of homological root homomorphisms $\bigl((\rootmap_2)_*,\dots,(\rootmap_{r-1})_*)$. Since the bundle 
$\pr:M(\A)\to M(\B)$ is the pullback of the bundle $\pi\colon\Conf_{n+1}(\C) \to \Conf_n(\C)$ along the root map $\rootmap=\rootmap_r \colon M(\B) \to \Conf_n(\C)$, we have commutative diagrams  of fundamental groups and associated LCS Lie algebras:

\begin{center}
\begin{tikzpicture}[scale=.75]
\node (t1) at (0,2){$1$}; 
\node (t2) at (2,2) {$F_n$};
\node (t3) at (4,2) {$G(\A)$};
\node (t4) at (6,2) {$G(\B)$};
\node (t5) at (8,2){$1$}; 
\node (b1) at (0,0){$1$}; 
\node (b2) at (2,0) {$F_n$};
\node (b3) at (4,0) {$P_{n+1}$};
\node (b4) at (6,0) {$P_{n}$};
\node (b5) at (8,0){$1$}; 
\node (tt1) at (10,2){$0$}; 
\node (tt2) at (12,2) {$\LL[n]$};
\node (tt3) at (14,2) {$\lie(\A)$};
\node (tt4) at (16,2) {$\lie(\B)$};
\node (tt5) at (18,2){$0$}; 
\node (bb1) at (10,0){$0$}; 
\node (bb2) at (12,0) {$\LL[n]$};
\node (bb3) at (14,0) {$\lie(P_{n+1})$};
\node (bb4) at (16,0) {$\lie(P_{n})$};
\node (bb5) at (18,0){$0$}; 
\draw[->] (t4) -- node[right] {$\rootmap_\sharp$} (b4);
\draw[->] (t3) -- (b3);
\draw[->] (t2) -- node[right] {$=$} (b2);
\draw[->] (t1) -- (t2); \draw[->] (t2) -- (t3); \draw[->] (t3) -- (t4); \draw[->] (t4) -- (t5);
\draw[->] (b1) -- (b2); \draw[->] (b2) -- (b3); \draw[->] (b3) -- (b4); \draw[->] (b4) -- (b5);
\draw[->] (tt4) -- node[right] {$\rootmap_*$} (bb4);
\draw[->] (tt3) -- (bb3);
\draw[->] (tt2) -- node[right] {$=$} (bb2);
\draw[->] (tt1) -- (tt2); \draw[->] (tt2) -- (tt3); \draw[->] (tt3) -- (tt4); \draw[->] (tt4) -- (tt5);
\draw[->] (bb1) -- (bb2); \draw[->] (bb2) -- (bb3); \draw[->] (bb3) -- (bb4); \draw[->] (bb4) -- (bb5);
\end{tikzpicture}
\end{center}
\noindent{Since} the Fadell-Neuwirth bundle admits a section, the bundle $M(\A)\to M(\B)$ does as well, and the rows of both diagrams are split exact.

Using the fact that $\lie(P_{n+1})$ is the semidirect product of $\lie(P_n)$ by $\LL[n]$ determined by the Lie homomorphism $\theta_n\colon \lie(P_n) \to \Der(\LL[n])$ given in \eqref{eq:infbraid}, the right-hand  pullback diagram of Lie algebras implies that $\lie(\A)$ is the semidirect product of $\lie(\B)$ by $\LL[n]$ determined by the composite
$\theta_n \circ \rootmap_* \colon \lie(\B) \to \Der(\LL[n])$. This completes the proof.
\end{proof}

\begin{remark} \label{rem:liepres}
The preceding result may be used to obtain a presentation for the LCS Lie algebra $\lie(\A)$. Denote the homology classes of the generators $y_{p,j}$ of $G(\A)$ by $e_{p,j}$. From the proof of \Cref{thm:LCSliealg}, these classes satisfy $[e_{p,i},e_{q,j}] = \ad\left((\rootmap_j)_*(e_{p,i})\right)(e_{q,j})$ in $\lie(\A)$, and all relations in $\lie(\A)$ are consequences of these.
Thus, $\lie(\A)$ is the quotient of the free Lie algebra generated by $\{e_{p,j} \mid 1\le j \le r,\ 1\le p\le n_j\}$ by the Lie ideal generated by
\[
 [e_{p,i},e_{q,j}] - \ad\left((\rootmap_j)_*(e_{p,i})\right)(e_{q,j}),\ 1\le i<j \le r,\ 1\le p\le n_i,\ 1\le q\le n_j.
\]
\end{remark}

\subsection{Cohomology ring}
We now turn our attention to the cohomology ring of the complement of a strictly supersolvable toric arrangement.

\begin{theorem} \label{thm:H*ring}
For $\A$ strictly supersolvable, the structure of the integral cohomology ring $H^*(M(\A);\Z)$ of the complement is determined by the sequence of homological root homomorphisms \eqref{eq:homroothom} 
and the infinitesimal pure braid relations. \end{theorem}
\begin{proof}
From \S\S\ref{subsec:tbundles}--\ref{subsec:group}, if $\A$ is strictly supersolvable of rank $r$, the fundamental group of the complement $G(\A)=\pi_1(M(\A))=\rtimes_{i=1}^{r} F_{n_i}$ is an almost-direct product of free groups (\Cref{cor:almostdirect}), and the complement $M(\A)$ is a $K(G(\A),1)$-space. Consequently, results of \cite{DCalmostdirect} may be used to determine the structure of $H^*(M(\A);\Z)=H^*(G(\A);\Z)$. We phrase the proof in terms of homology and cohomology of groups.
For the remainder of the paper, we suppress coefficients when considering (co)homology with integer coefficients, e.g., $H^*(G(\A))=H^*(G(\A);\Z)$.

Since $G(\A)$ is an almost-direct product of free groups, the integral homology groups are free abelian, with ranks given by
\[
\sum_{i=1}^r \mathrm{rank}\,H_i(G(\A)) \cdot t^i = \prod_{j=1}^r (1+n_jt),
\]
see \cite[Corollary 3.2]{CSchain}. In particular, $H_1(G(\A))$ has rank $N=n_1+\dots+n_r$ and 
$H_2(G(\A))$ has rank $\sum_{i<j} n_i n_j$. Note that $N=r+|\A|$, and observe that $\mathrm{rank}\,H_2(G(\A))$ is equal to the number of
relations recorded in \eqref{eq:pi1rels}. As shown in \cite[\S2]{DCalmostdirect}, the abelianization map $\ab\colon G(\A) \to \Z^N$ induces isomorphisms $a_*\colon H_1(G(\A))\to H_1(\Z^N)$, 
$a^*\colon H^1(\Z^N) \to H^1(G(\A))$, a monomorphism  $\ab_*\colon H_2(G(\A)) \to H_2(\Z^N)$, and an epimorphism $\ab^*\colon H^2(\Z^N) \to H^2(G(\A))$. 

Denote the integral exterior algebra $H^*(\Z^N)$ by $\EE$, and let $\II$ be the ideal in $\EE$ generated by \[\ker\bigl(\ab^*\colon H^2(\Z^N)\to H^2(G(\A))\bigr).\] Similarly, let $\EE_\QQ=H^*(\Z^N;\QQ)$ and let $\II_\QQ$ be the ideal in $\EE_\QQ$ corresponding to $\II\subset \EE$. 
By \cite[Theorem 3.1]{DCalmostdirect}, we have $H^*(G(\A);\QQ) \cong {\EE_\QQ}/{\II_\QQ}$. Since $H^*(G(\A))$ is torsion-free (and is finitely generated), it follows that $H^*(G(\A)) \cong {\EE}/{\II}$.
As the map $\ab^*\colon H^2(\Z^N)\to H^2(G(\A))$ is dual to $\ab_*\colon H_2(G(\A)) \to H_2(\Z^N)$, to prove the theorem, it suffices to show that the latter is determined by the sequence of homological root homomorphisms 
and the infinitesimal pure braid relations. 

Recall that $G(\A)=\rtimes_{j=1}^{r} F_{n_j}$ has generators $\ys_{p,j}$, $1\le p \le n_j$, $1\le j \le r$, and relations given by \eqref{eq:pi1rels}. 
Write a representative such relation as $\xs^{-1} \ys \xs=\phi(\xs)(\ys)$, where $\xs=\ys_{p,i}$, $\ys=\ys_{q,j}$, $i<j$, and $\phi$ is the composition of the (faithful) Artin representation and the homomorphism induced by the root map $\rootmap_j$. The pure braid $\zs=(\rootmap_j)_\sharp(\xs)$ acts by conjugation, so we have $\phi(\xs)(\ys)=\zs^{-1}\ys \zs = \ws \ys \ws^{-1}$, where $\ws,\ys \in F_{n_j}$. 
Since
$\ys\xs=\xs\ys[\ys^{-1},\ws]=\xs\ys\cdot \ys^{-1} [\ws,\ys]{\ys} = \xs\ys\cdot \bigl[\ys^{-1},[\ws,\ys]\bigr] \cdot [\ws,\ys]$, 
this representative relation can be rewritten in the form 
\begin{equation} \label{eq:reprel}
\ys\xs
= \xs\ys\cdot \bigl[\ys^{-1},[\ws,\ys]\bigr] \cdot [\ws,\ys],
\end{equation}
as in \cite[Proposition 2.2]{DCalmostdirect}.

The homology classes $e_{p,j}$ of the generators $\ys_{p,j}$ of $G(\A)$ form a basis for $H_1(G(\A))=H_1(\Z^N)$. Identifying $H_2(\Z^N)=\Z^{\binom{N}{2}}$ with the second graded piece of the 
exterior algebra $\EE$, this group has basis $e_{p,i}e_{q,j}$ where $1\le i\le j \le r$, $1\le p \le n_i$, $1\le q \le n_j$, and $p<q$ if $i=j$. 
Since, as noted above, the (free abelian) group $H_2(G(\A))$ has rank $\sum_{i<j} n_in_j$, this group has generators in correspondence with (the reformulations \eqref{eq:reprel} of) the relations \eqref{eq:pi1rels}. If $\mathbf{r}$ denotes the (generator corresponding to the) representative relation \eqref{eq:reprel} above, with $\xs=\ys_{p,i}$, $\ys=\ys_{q,j}$, it follows from \cite[\S2]{DCalmostdirect} that $\ab_*(\mathbf{r}) = e_{p,i}e_{q,j} + {\mathsf{W}} e_{q,j}$, where ${\mathsf{W}}=\ab(\ws)$ is the {image of $\ws$ under the abelianization map.}

The relation \eqref{eq:reprel} also gives rise to a relation in the LCS Lie algebra $\lie(\A)$ of $G(\A)$, as in \cite[Lemma 2.3.4]{cohen}. Rewriting the relation $\mathbf{r}$ as
\[
1_{G(\A)} = [\xs\ws,\ys]=\left[\xs,[\ws,\ys]\right] \cdot [\ws,\ys] \cdot [\xs,\ys], 
\]
we have $0=[{\mathsf{W}},e_{q,j}] + [e_{p,i},e_{q,j}] = [e_{p,i}+{\mathsf{W}},e_{q,j}]$ in $\lie(\A)$. Since these (defining) relations in $\lie(\A)$ are determined by  
the sequence of homological root homomorphisms and the infinitesimal pure braid relations by \Cref{thm:LCSliealg}, so is the map $\ab_*\colon H_2(G(\A)) \to H_2(\Z^N)$, as required.
\end{proof}

\begin{corollary} \label{cor:koszul}
    For $\A$ strictly supersolvable, the integral cohomology ring $H^*(M(\A))\cong \EE/\II$ is a Koszul algebra. 
\end{corollary}
\begin{proof}[Sketch of proof] It suffices to show that the ideal $\II$ in the exterior algebra $\EE$ has a quadratic Gr\"obner basis. This is established over the rationals in 
\cite[\S 3]{DCalmostdirect}. We sketch how the discussion there may be used to establish the integral case.

An explicit basis $\mathcal{J}$ for $\ker\bigl( \ab^*\colon H^2(\Z^N) \to H^2(G(\A))\bigr)$ is given in \cite[Corollary 2.5]{DCalmostdirect}, and it is subsequently shown that $\mathcal{J}$ is a Gr\"obner basis for the ideal $\II_\QQ=\left\langle \ker\bigl(\ab^*\colon H^2(\Z^N;\QQ)\to H^2(G(\A);\QQ)\bigr)\right\rangle$ in the degree-lexicographic monomial order. The proof of 
\cite[Lemma 3.3]{DCalmostdirect} establishing the latter includes the construction of a basis 
for the entire ideal $\II_\QQ$. This basis is integral by inspection, so is a basis for the ideal $\II$. 
Using this basis and the same monomial order, one can check that the initial ideal of $\II$ is generated by the initial terms of the elements of $\mathcal{J}$. Hence, $\mathcal{J}$ is a Gr\"obner basis for $\II$.
\end{proof}

\begin{remark}
Turning briefly to rational homotopy theory, Dupont established formality of toric arrangement complements  \cite[Theorem 1.3]{dupont}. This has implications via the work of Papadima and Yuzvinsky \cite{PY}. When $\A$ is strictly supersolvable, 
Koszulity of the rational cohomology ring in \Cref{cor:koszul} implies that the Bousfield--Kan rational completion of $M(\A)$ is $K(\pi,1)$. Moreover, the Koszul dual of $H^*(M(\A);\QQ)$ is the universal enveloping algebra of the rational LCS Lie algebra of $\pi=\pi_1$, and this relationship yields an alternate proof of the LCS formula in \cite[Theorem D]{BD}.
\end{remark}

\subsection{Topological complexity} Let $X$ be a path-connected topological space with the homotopy type of a finite cell complex, and let $X^I$ denote the space of all continuous paths $\gamma\colon I=[0,1] \to X$. 
\begin{definition}
The \textbf{topological complexity} of $X$, denoted $\tc(X)$, is the sectional category of the fibration $\pi\colon X^I \to X \times X$, $\gamma \mapsto (\gamma(0),\gamma(1))$, sending a path to its endpoints. That is, $\tc(X)=\secat(\pi\colon X^I \to X\times X)$ is the smallest positive integer $k$ for which $X\times X=U_1\cup U_2 \cup \dots \cup U_k$ where each $U_i$ is open and there is a continuous section $s_i\colon U_i \to X^I$ of the path space fibration, $\pi\circ s_i=\id_{U_i}$, for each $i$, $1\le i \le k$.     
\end{definition}

The homotopy-type invariant $\tc(X)$, introduced by Farber \cite{farber}, is motivated by the motion planning problem from robotics. This notion may be extended to a discrete group $G$ by defining $\tc(G)$ to be the topological complexity of an Eilenberg-Mac\,Lane space of type $K(G,1)$.

\begin{theorem} \label{thm:TC}
If $\A$ is a strictly supersolvable toric arrangement of rank $r$ in $(\C^\times)^d$, then the topological complexity of the complement is $\tc(M(\A))=d+r+1$.
\end{theorem}
\begin{proof}
As noted in \Cref{rem:essential}, there is an essential toric arrangement $\A'$ in $(\C^\times)^r$ so that $M(\A) \cong M(\A')\times (\C^{\times})^{d-r}$. 
Since $\A$ is strictly supersolvable, so is $\A'$. By \Cref{cor:almostdirect}, the group $G(\A')=\pi_1(M(\A')) \cong \rtimes_{j=1}^r F_{n_j}$ is an almost-direct product of free groups. The fact that $\A'$ is essential implies that the ranks of these free groups satisfy $n_j \ge 2$ for each $j$, $1\le j\le r$. By \cite[Theorem 4.2]{DCalmostdirect}, we have $\tc(G(\A')\times \Z^m)=2r+m+1$ for any non-negative integer $m$. Since $G(\A)=\pi_1(M(\A)) \cong G(\A')\times \Z^{d-r}$ and $M(\A)$ is a $K(G(\A),1)$-space so that $\tc(M(\A))=\tc(G(\A))$, taking $m=d-r$ completes the proof.
\end{proof}

\section{Rank two circuits} \label{subsec:rank2circ}
We illustrate results from \Cref{sec:toric} and \Cref{sec:ssta} using a class of strictly supersolvable rank two toric arrangements, namely, rank two circuits. 

For integers $k,m_1,m_2$ with $k>0$ and $m_2-m_1=m>0$, let $n=km=k(m_2-m_1)$ and consider the toric arrangements $\CC$ and $\CC_{n,m}$ in $(\C^\times)^2$ with character matrices
\[
\begin{pmatrix}n&m_1&m_2\\0&1&1\end{pmatrix} \quad \text{and} \quad \begin{pmatrix}n&-m&0\\0&1&1\end{pmatrix}.
\]
The maps $M(\CC) \to M(\CC_{n,m})$, $(x,y) \mapsto (x,x^{m_2}y)$ and $M(\CC_{n,m}) \to M(\CC)$, $(x,y) \mapsto (x,x^{-m_2}y)$ are homeomorphisms, so we work exclusively with $\CC_{n,m}$.
The arrangement $\CC_{n,m}$ in $(\C^\times)^2 \subset \C^2$, given by the vanishing of the polynomial $x(x^n-1)y(y-x^m)(y-1)$, is strictly supersolvable over the arrangement $\B$ in $\C^\times\subset \C$, given by the vanishing of $x(x^n-1)$. 

\subsection{Fundamental group} \label{subsec:r2pi1}
By \Cref{prop:coef/root}\,\eqref{item:root1}, the bundle $M(\CC_{n,m}) \to M(\B)$ is equivalent to the pullback of the Fadell-Neuwirth bundle $\Conf_4(\C) \to \Conf_3(\C)$ along the map $\rootmap\colon M(\B) \to \Conf_3(\C)$ given by $\rootmap(x)=(0,x^m,1)$. We determine the map on fundamental groups induced by $\rootmap$. With $\mu=\exp(2\pi \ii/n)$ where $\ii=\sqrt{-1}$, the fundamental group of $M(\B)=\C\smallsetminus\{0,1,\mu,\dots,\mu^{n-1}\}$ is free on $n+1$ generators. Fix $\epsilon>0$ small, and fix the basepoint $\ast=1-\epsilon$ in $M(\B)$. Let $\ell(t)=1-\epsilon \exp(2 \pi \ii t)$, $0\le t \le 1$ be a loop about $1$ based at $\ast$, and for $1\le j \le n$, let $f_j(t)=(1-\epsilon)\exp(2 \pi \ii t)$, $0\le t \le j/n$, be the circular arc from $*$ to $\mu^j(1-\epsilon)$. 
Note that $\mu^j \ell(t)$ is a loop about $\mu^j$ based at $\mu^j(1-\epsilon)$, and that $f_n(t)$ is a loop about $0$ based at $\ast$. Loops based at $\ast$ representing the generators of $\pi_1(M(\B))=F_{n+1}$ are then given by $\gamma_0(t)=f_n(t)$, $\gamma_j(t)=f_j(t) 
\centerdot \mu^j\ell(t) \centerdot 
\bar f_j(t)$ for $1 \le j \le n-1$,
where $\bar f_j(t):=f_j(1-t)$ denotes the reverse path, 
$f(t) \centerdot g(t)$ 
denotes concatenation, and $\gamma_n(t)=\ell(t)$.
The case $n=6$, $m=3$ is illustrated in \Cref{fig:n6m3}.
\begin{figure}[h] 
\begin{tikzpicture} 
\draw[color=blue] (1,0) circle [radius=.3]; 
\draw[color=blue] (-.5,.866) circle [radius=.3]; 
\draw[color=blue] (-.5,-.866) circle [radius=.3]; 
\draw[color=purple] (.5,.866) circle [radius=.3]; 
\draw[color=purple] (.5,.866) circle [radius=.3]; 
\draw[color=purple] (.5,-.866) circle [radius=.3]; 
\draw[color=purple] (-1,0) circle [radius=.3]; 
\node at (0,0) {\small{$0$}};
\node[color=blue] at (1,0) {\small{$1$}};
\node[color=purple] at (.5,.866) {\small{$\mu$}};
\node[color=blue] at (-.5,.866) {\small{$\mu^2$}};
\node[color=purple] at (.5,-.866) {\small{$\mu^5$}};
\node[color=blue] at (-.5,-.866) {\small{$\mu^4$}};
\node[color=purple] at (-1,0) {\small{$\mu^3$}};
\draw[color=black] (0,0) circle [radius=.7]; 
\node[color=black,anchor=center,inner sep=.5,fill=white] at (.7,0) {$*$};
\end{tikzpicture} 
\caption{Loops in $M(\B)$ when $n=6$, $m=3$.} \label{fig:n6m3} 
\end{figure}

The fundamental group of $\Conf_3(\C)$ (with basepoint $\rootmap(\ast)=(0,r,1)$, where $0<r=(1-\epsilon)^m<1$) is the 3-strand pure braid group $P_3=\langle a_{1,2}^{},a_{1,3}^{},a_{2,3}^{}\mid a_{1,2}^{}a_{1,3}^{}a_{2,3}^{}= a_{1,3}^{}a_{2,3}^{}a_{1,2}^{}=a_{2,3}^{}a_{1,2}^{}a_{1,3}^{}\rangle$. The pure braids $a_{1,2}^{}$ and $a_{2,3}^{}$ may be represented by loops $(0,r\exp(2\pi \ii \theta),1)$ and $(0,1-(1-r)\exp(2\pi \ii \theta),1)$, $0\le \theta \le 1$, respectively. (An explicit representative of $a_{1,3}^{}$ will not be needed in the following calculations.) The map $\rootmap_\sharp \colon \pi_1(M(\B)) \to P_3$ induced by $\rootmap$ is given by
\begin{equation} \label{eq:gonpi1a}
\begin{aligned} 
\rootmap_\sharp([\gamma_0])&=[\rootmap\circ f_n]= [(0,r\exp(2\pi \ii m t),1)] = a_{1,2}^m,\\
\rootmap_\sharp([\gamma_n])&= [\rootmap\circ \ell]= [(0,(1-\epsilon\exp(2 \pi \ii t))^m,1)] = [(0,1-r\exp(2\pi \ii t),1)] = a_{2,3}^{},
\end{aligned}
\end{equation}
and, recalling that $n=km$, for $1 \le j < n$,
\begin{equation} \label{eq:gonpi1b}
\rootmap_\sharp([\gamma_j])= \begin{cases}
a_{1,2}^q a_{2,3}^{}a_{1,2}^{-q}&\text{if $j=qk$, so that $\mu^j$ is an $m$-th root of unity,} \\
1 &\text{otherwise}.
\end{cases}
\end{equation}
A few details of these calculations follow. It will be enough to take $\epsilon < \min\{\frac{1}{2},\sin(\pi/2mn)\}$. 

For the second equation in \eqref{eq:gonpi1a} we show that for every $t$ the segment $S(t)$ between $\ell(t)^m$ and $(1-r\exp(2\pi\ii t))$ is contained in $\C\smallsetminus\{0,1\}$, so that the map $H\colon I\times I \to \Conf_3(\C), (s,t) \mapsto (0,h(s,t),1)$ with $h(s,t)= s(\ell(t))^m + (s-1)(1-r\exp(2\pi\ii t))$ is a well-defined path homotopy. A straightforward check shows that if $t=0$ then $S(t)\subseteq ]0,1[$, and if $t=\frac{1}{2}$ then $S(t)\subseteq ]1,\infty[$. If $0<t<\frac{1}{2}$ then $\Im(1-r\exp(2\pi\ii t))<0$ and the condition on $\epsilon$ implies 
$
0>\arg(\ell( t))> -\frac{\pi}{m}
$
so that $\Im (\ell(t)^m)<0$ as well and thus by convexity $S(t)\subseteq \R + \ii \R_{<0}\subseteq \C\smallsetminus \{0,1\}$. The case $\frac{1}{2} < t < 1$ is analogous.

For \eqref{eq:gonpi1b} let $1 \le j \le n$ and consider two cases. First, if $j=kq$ for some integer $q$, let $\theta=mt$ and observe that the path $\rootmap\circ f_j(t)$, $0\le t \le j/n$, is in fact the loop $(0,r \exp(2 \pi \ii \theta),1)$, $0 \le \theta \le q$, which represents $a_{1,2}^q$. For such $j$ note also that $\rootmap\circ (\mu^j \ell(t))=\rootmap\circ \ell(t)$ represents $a_{2,3}^{}$. If on the other hand $k$ does not divide $j$, using the condition on $\epsilon$ one shows that  the loop $(\mu^j \ell(t))^m$ is contained in a ``sector'' $U$ of amplitude $<\pi/n$ around the nontrivial $n$-th root of unity $\mu^{jm}\neq 1$. Since $U$ is contractible and misses $0$ and $1$, any path contained in it  is nullhomotopic in $\C\smallsetminus\{0,1\}$. This shows that $\rootmap\circ \ell(t)$ is nullhomotopic in $\Conf_3(\C)$.

Since the 
bundle $M(\CC_{n,m}) \to M(\B)$ is equivalent to the pullback along $\rootmap$ of the 
bundle $\Conf_4(\C) \to \Conf_3(\C)$, with fiber $\C\smallsetminus\{\text{3 points}\}$, the fundamental group $\pi_1(M(\CC_{n,m})) \cong F_3 \rtimes_{\alpha} F_{n+1}$ is the semidirect product of free groups determined by the homomorphism 
$\phi=\hat \alpha_3 \circ \rootmap_\sharp \colon F_{n+1} \to \Aut(F_3)$, where $\hat \alpha_3\colon P_3 \to \Aut(F_3)$ is (the restriction of) the Artin representation.  
Denoting the generators of $F_{n+1}$ by $\xs_i^{}=[\gamma_i]$, $0\le i \le n$, and those of $F_3$ by $\ys_1^{},\ys_2^{},\ys_3^{}$, the group $\pi_1(M(\CC_{n,m}))$ has presentation
\begin{equation} \label{eq:pi1pres}
\pi_1(M(\CC_{n,m})) = \langle \xs_0^{}, \xs_1^{}, \dots, \xs_n^{}, \ys_1^{},\ys_2^{},\ys_3^{} \mid \xs_i^{-1}\ys_j^{}\xs_i^{} = \phi(\xs_i^{})(\ys_j^{}),\ 0\le i \le n,\ 1 \le j \le 3 \rangle.
\end{equation}

This may be made explicit using the Artin representation, see \S\ref{subsec:Artin}. 
For $q \in \Z$, one checks that $a_{1,2}^q(\ys_i^{})=(\ys_1^{}\ys_2^{})^q\ys_i^{}(\ys_1^{}\ys_2^{})^{-q}=[(\ys_1^{}\ys_2^{})^q,\,\ys_i^{}]\cdot \ys_i^{}$ for $i=1,2$, and $a_{1,2}^q(\ys_3^{})=\ys_3^{}$. Using this, one can show that
\[
(a_{1,2}^qa_{2,3}^{}a_{1,2}^{-q})(\ys_i)=\begin{cases}
[\ws_q^{}(\ys_1^{}\ys_2^{})^{-q},\, \ys_1^{}]\cdot \ys_1^{}&\text{if $i=1$,}\\
[\ws_q^{}(\ys_1^{}\ys_2^{})^{-q}\ys_2^{}(\ys_1^{}\ys_2^{})^{q}\ys_3^{}(\ys_1^{}\ys_2^{})^{-q},\, \ys_2^{}]\cdot \ys_2^{}&\text{if $i=2$,} \\
[(\ys_1^{}\ys_2^{})^{-q}\ys_2^{}(\ys_1^{}\ys_2^{})^{q},\ys_3^{}] \cdot \ys_3^{}&\text{if $i=3$,}\\
\end{cases}
\]
where 
\[
\begin{aligned}
\ws_q&=(a_{2,3}^{}a_{1,2}^{-q})\bigl((\ys_1^{}\ys_2^{})^q\bigr)=a_{1,2}^{-q}\bigl(a_{2,3}^{}\bigl((\ys_1^{}\ys_2^{})^q\bigr)\bigr)\\
&=
\bigl(\ys_1^{}\ys_2^{}\ys_3^{}(\ys_1^{}\ys_2^{})^{-q}\ys_2^{}(\ys_1^{}\ys_2^{})^q\ys_3^{-1}(\ys_1^{}\ys_2^{})^{-q}\ys_2^{-1}(\ys_1^{}\ys_2^{})^q\bigr)^q.
\end{aligned}
\]
Rewriting \eqref{eq:gonpi1a} and \eqref{eq:gonpi1b} as 
\begin{equation} \label{eq:gonpi1}
\phi(\xs_0^{})=a_{1,2}^m,\quad\phi(\xs_n^{})=a_{2,3}^{},\quad
\phi(\xs_j^{})=\begin{cases}
a_{1,2}^qa_{2,3}^{}a_{1,2}^{-q}&\text{when $1\le j=qk \le n-1$,}\\ 1&\text{when $1\le j \neq qk \le n-1$,}\end{cases}
\end{equation}
the calculations above may be used to express $\xs_i^{-1}\ys_j^{}\xs_i^{} = \phi(\xs_i^{})(\ys_j^{})$
in terms of the generators $\xs_i^{}, \ys_j^{}$.

\begin{example} \label{ex:n6m3} Consider the case $m=3$, $n=km=6$. Loops in the base of the strictly supersolvable bundle $M(\CC_{6,3})\to M(\B)$ are depicted in Figure \ref{fig:n6m3}. The discussion above yields a presentation for the group $\pi_1(M(\CC_{6,3}))$ with generators $\xs_0^{},\xs_1^{},\dots,\xs_6^{},\ys_1^{},\ys_2^{},\ys_3^{}$ and relations $\xs_i^{-1}\ys_j^{}\xs_i^{}=w_{i,j}^{} \ys_j^{}w_{i,j}^{-1}$. Note that $w_{i,j}^{} \ys_j^{}w_{i,j}^{-1}=[w_{i,j},\,\ys_j^{}]\cdot \ys_j^{}$.  For $i=1,3,5$, we have $w_{i,j}=1$.
For other values of $i$, writing $u^v=v^{-1}uv$, we have
\[
\begin{array}{lll}
w_{0,1}=(\ys_1^{}\ys_2^{})^3,\quad&w_{0,2}=(\ys_1^{}\ys_2^{})^3,\quad&w_{0,3}=1,\\
w_{2,1}=\ws_1^{}(\ys_1\ys_2)^{-1},&w_{2,2}=\ys_3^{(\ys_1\ys_2)^{-1}},&w_{2,3}=\ys_2^{(\ys_1\ys_2)},\\
w_{4,1}=\ws_2^{}(\ys_1\ys_2)^{-2},\quad&w_{4,2}=\ws_2^{}\ys_2^{(\ys_1\ys_2)^2}\ys_3^{}(\ys_1\ys_2)^{-2},\quad&w_{4,3}=\ys_2^{(\ys_1\ys_2)^2},\\
w_{6,1}=1,&w_{6,2}=\ys_2^{}\ys_3^{},&w_{6,3}=\ys_2^{}\ys_3^{}.
\end{array}
\]
\end{example}

\subsection{Lower central series Lie algebra} \label{subsec:r2lcs}
Passing to (integral) homology, let $X_j=[\xs_j]$, $0 \le j \le n$ and $A_{i,j}=[a_{i,j}]$, $1\le i<j\le 3$ denote the generators of $H_1(M(\B))\cong\Z^{n+1}$ and $H_1(\Conf_3(\C))\cong \Z^{3}$, respectively. From \eqref{eq:gonpi1}, the homological root homomorphism $\rootmap_*\colon H_1(M(\B)) \to H_1(\Conf_3(\C))$ is then given by
\[
\rootmap_*(X_0)=m A_{1,2}, \quad \rootmap_*(X_j)=\begin{cases}A_{2,3}&\text{if $j=qk$, so that $\xi^j$ is an $m$-th root of unity,} \\
0 &\text{otherwise}.
\end{cases}
\]

By Theorem \ref{thm:LCSliealg}, the integral lower central series (LCS) 
Lie algebra $\lie(\CC_{n,m})$ of $\pi_1(M(\CC_{n,m})$ is the semidirect product of the free Lie algebra $\LL[n+1]$ (generated by $X_j$, $0\le j\le n$) by the free Lie algebra $\LL[3]$ (generated by $Y_i=A_{i,4}$, $1\le i \le 3$) determined by the Lie homomorphism $\Theta=\theta_3 \circ \rootmap_* \colon \LL[n+1] \to \Der(\LL[3])$, where $\theta_3(A_{i,j})=\ad(A_{i,j})$. Using the above description of the map $\rootmap_*$, we have ${[}X_0,Y_i]=m \ad(A_{1,2})(Y_i)$, ${[}X_j,Y_i]=\ad(A_{2,3})(Y_i)$ if $j=qk$, and ${[}X_j,Y_i]=0$ if $j\neq qk$, yielding
\[
\begin{array}{lll}
[X_0,Y_1]=m[Y_1,Y_2],\quad & [X_0,Y_2]=m[Y_2,Y_1],\quad & [X_0,Y_3]=0,\\ 
{[}X_j,Y_1]=0,\quad & [X_j,Y_2]=[Y_2,Y_3],\quad & [X_j,Y_3]=[Y_3,Y_2], \quad \text{if $j=kq$.} 
\end{array}
\]
Consequently, the LCS Lie algebra $\lie$ may be realized as the quotient of the free Lie algebra $\LL[n+4]$ (generated by $X_j$, $0\le j\le n$, and $Y_1,Y_2,Y_3$) by the Lie ideal $\mathcal J$ generated by
\[
[X_0,Y_1]-m[Y_1,Y_2],\ [X_0,Y_2]+m[Y_1,Y_2],\ [X_0,Y_3],\ [X_j,Y_1],\ [X_j,Y_2] - c_j[Y_2,Y_3],\ [X_j,Y_3] + c_j[Y_2,Y_3],
\]
where $1\le j \le n$, $c_j=1$ if $j=kq$, and $c_j=0$ if $j \neq kq$.

\begin{remark}
The space $M(\CC_{n,m})$ is a $K(G,1)$-space for the almost-direct product of free groups $G=\pi_1(M(\CC_{n,m}))=F_3 \rtimes_{\phi} F_{n+1}$. 
The relations in the presentation \eqref{eq:pi1pres}, resp., the generators of the Lie ideal $\mathcal J$ above, are in correspondence with a basis 
$\{r_{i,j}, 0 \le i  \le n, 1\le j\le 3\}$ for $H_2(G)=H_2(M(\CC_{n,m})) \cong \Z^{3n+3}$. Furthermore, the injective map $\ab_* \colon H_2(G) \to H_2(\Z^{n+4})$ induced by the abelianization $\ab \colon G \to \Z^{n+4}$ may be determined from the generators of $\mathcal J$ as in the proof of \Cref{thm:H*ring}.

Fixing generators $e_j$, $0\le j \le n$, and $f_1,f_1,f_3$ for $H_1(\Z^{n+4})$, the group $H_2(\Z^{n+4})$ may be identified with the second graded piece of the exterior algebra $\bigwedge H_1(\Z^{n+4})$, generated by $e_ie_j, e_if_k, f_kf_l$ ($i<j$, $k<l$). The map $\mathfrak{a}_* \colon H_2(G) \to H_2(\Z^{n+4})$ is then given by
\begin{equation} \label{eq:H2map}
\begin{array}{lll}
\ab_*(r_{0,1})=e_0f_1-mf_1f_2,\quad &\ab_*(r_{0,2})=e_0f_2+mf_1f_2,\quad &\ab_*(r_{0,3})=e_0f_3,\\
\ab_*(r_{j,1)}=e_jf_1,\quad &\ab_*(r_{j,2})=e_jf_2-c_jf_2f_3,\quad &\ab_*(r_{j,3})=e_jf_3+c_jf_2f_3,
\end{array}
\end{equation}
where, as above, $1\le j \le n$, $c_j=1$ if $j=kq$, and $c_j=0$ if $j \neq kq$.
\end{remark}

\begin{example} \label{ex:n6m3part2} We continue with the case $m=3$, $n=6$. The LCS Lie algebra of $G=\pi_1(M(\CC_{6,3}))$ is the quotient of the free Lie algebra $\LL[10]$, generated by $X_0,X_1,\dots,X_6,Y_1,Y_2,Y_3$, by the Lie ideal ${\mathcal J}$ generated by 
\[
\begin{array}{llll}
[X_0,Y_1]-3[Y_1,Y_2],\quad &[X_0,Y_2]+3[Y_1,Y_2],\quad &[X_0,Y_3],\\
{[}X_j,Y_1],\quad &[X_j,Y_2],\quad &[X_j,Y_3],\quad &\text{for $j=1,3,5$,}\\
{[}X_j,Y_1],\quad &[X_j,Y_2]-[Y_2,Y_3],\quad &[X_j,Y_3]+[Y_2,Y_3],\quad &\text{for $j=2,4,6$.}\\
\end{array}
\]
The map $\ab_2 \colon H_2(G) \to H_2(\Z^{10})$ is given by
\[\begin{array}{llll}
\ab(r_{0,1})=e_0f_1-3f_1f_2,\quad &\ab(r_{0,2})=e_0f_2+3f_1f_2,\quad &\ab(r_{0,3})=e_0f_3,\\
\ab(r_{j,1})=e_jf_1,\quad &\ab(r_{j,2})=e_jf_2,\quad &\ab(r_{j,3})=e_jf_3,\quad &\text{for $j=1,3,5$,}\\
\ab(r_{j,1})=e_jf_1,\quad &\ab(r_{j,2})=e_jf_2-f_2f_3,\quad &\ab(r_{j,3})=e_jf_3+f_2f_3,\quad &\text{for $j=2,4,6$.}\\
\end{array}
\]
\end{example}

\subsection{Cohomology ring}\label{sec:cmn:cohom} 
As  
discussed in the proof of \Cref{thm:H*ring}, the integer cohomology ring of the $K(G,1)$-space $M(\CC_{n,m})$ 
is isomorphic to $\EE/\II$, where $\EE=\bigwedge H^1(\Z^{n+4})$ and \[\II=\langle\ker(\ab^*\colon H^2(\Z^{n+4}) \to H^2(G)\rangle.\] Denote the generators of $H^1(\Z^{n+4})=\Hom(H^1(\Z^{n+4}),\Z)$ by the same symbols as those of $H_1(\Z^{n+4})$, so that $\EE$ is the exterior algebra on $e_0,e_1,\dots,e_n,f_1,f_2,f_3$. 
Calculating with \eqref{eq:H2map} and recalling that $n=km$, we obtain $H^*(M(\CC_{n,m})) \cong \EE/\II$, where 
\[
\II = \ker(\ab_2^*) = \left\langle e_ie_j, 0\le i<j\le n,\ f_1f_2+me_0(f_1-f_2),\ f_1f_3,\ f_2f_3+\sum_{q=1}^m e_{kq}(f_2-f_3)\right\rangle.
\]
\begin{example} \label{ex:n6m3part3} Returning to the case $m=3$, $n=6$, we have $H^*(M(\CC_{6,3})) \cong \EE/\II$, where $\EE$ is the exterior algebra on $e_0,e_1,\dots,e_6,f_1,f_2,f_3$, and $\II$ is the ideal generated by $e_ie_j$, $0\le i<j\le 6$, $f_1f_2+3e_0(f_1-f_2)$, $f_1f_3$, and $f_2f_3+(e_2+e_4+e_6)(f_2-f_3)$.
\end{example}

We compare our presentation of the cohomology ring for toric arrangements with others in the literature.
One presentation, from \cite[Theorem 7.4]{CDDMP} and \cite[Theorem 5.9]{BPP}, is for the integral (hence also rational) cohomology, and the second, from \cite[Theorem 6.13]{CDDMP}, is a different presentation for the rational cohomology. 
Below, we describe each of these presentations in the case of $M(\CC_{n,m})$, after some simplification by removing redundant generators, and provide explicit isomorphisms to our presentation.

First, consider the exterior $\Z$-algebra $\EE_1$ on generators $z_1$, $z_2$, $w_1$, $w_2$, and $w_{0,j}$ for $j=1,\dots,n$, and the ideal 
\[\II_1=\left\langle
\begin{array}{lll}
(-mz_1+z_2)w_1, & z_2w_2, \qquad & z_1w_{0,j} \text{ for } 0\leq j\leq n, \\
w_{0,i}w_{0,j} \text{ for } 1\leq i<j\leq n, \ 
& &
w_1w_2+mz_1w_2+\sum_{q=1}^m w_{0,qk}(w_1-w_2)
\end{array}
\right\rangle\]
As in \cite[Theorem 7.4]{CDDMP}, we view the generators of $\EE_1$ as represented by the differential forms below, with $\II_1$ capturing precisely the relations satisfied by these forms, so that the quotient $\EE_1/\II_1$ is isomorphic to $H^*(M(\CC_{n,m}))$. With $\ii=\sqrt{-1}$, write  
\[\begin{array}{lll}
z_1 = \frac{1}{2\pi\ii}\mathrm{dlog}(x), \quad & z_2 = \frac{1}{2\pi\ii}\mathrm{dlog}(y), \\
w_1 = \frac{1}{2\pi\ii}\mathrm{dlog}(1-x^{-m}y), \quad & w_2 = \frac{1}{2\pi\ii}\mathrm{dlog}(1-y), \quad & w_{0,j} = \frac{1}{2\pi\ii}\mathrm{dlog}(1-\zeta^j x) \text{ for } 1\leq j\leq n.
\end{array}\]
Comparing this presentation to the one we derived above, there is an explicit isomorphism from $\EE_1/\II_1$  to $\EE/\II$ given by
\[\begin{array}{llll}
z_1 \mapsto e_0, \qquad & 
z_2 \mapsto f_1, \qquad \\
w_1 \mapsto -me_0+f_2, \qquad
& w_2 \mapsto f_3,  \qquad
& w_{0,j} \mapsto e_j 
&  \text{ for } 1\leq j \leq n.\\
\end{array}\]

Alternatively, using rational coefficients, consider the exterior $\QQ$-algebra $\EE_2$ on generators $\psi_0$, $\psi_1$, $\psi_2$, $\ww_1$, $\ww_2$, and $\ww_{0,j}$ for $1\leq j\leq n$, and the ideal 
\[\II_2=\left\langle
\begin{array}{lll}
\psi_0-k\psi_1+k\psi_2, & \psi_1\ww_1, & \psi_2\ww_2, \\
\psi_0\ww_{0,j} \text{ for } 1\leq j \leq n, \  &
\ww_{0,i}\ww_{0,j} \text{ for } 1\leq i<j\leq n, \  &
\ww_1\ww_2-\psi_1\psi_2+\sum_{q=1}^m\ww_{0,qk}(\ww_1-\ww_2)
\end{array}
\right\rangle\]
 By \cite[Theorem 6.13]{CDDMP}, the quotient $\EE_2/\II_2$ is isomorphic to $H^*(M(\CC_{n,m});\QQ)$. The generators 
  are represented by 
 differential forms distinct from those above (see \cite[Definition 2.20]{CDDMP}). Comparing to our presentation, there is an explicit isomorphism from $\EE_2/\II_2$ to the rationalization of $\EE_1/\II_1$, hence also that of $\EE/\II$, given by
\[\begin{array}{lll}
\psi_0\mapsto nz_1, \qquad
& \psi_1\mapsto -mz_1+z_2,\qquad
& \psi_2\mapsto z_2,
\\
\ww_1\mapsto 2w_1+mz_1-z_2, \quad 
& \ww_2 \mapsto 2w_2-z_2, \quad 
& \ww_{0,j} \mapsto 2w_{0,j}-z_1 \text{ for } 1\leq j\leq n.
\end{array}\]

\section{Homological root homomorphisms} \label{sec:hrm}
As illustrated in \S\ref{sec:ssta} and \S\ref{subsec:rank2circ}, for a strictly supersolvable toric arrangement $\A$, determining the sequence of homological root homomorphisms \eqref{eq:homroothom} yields explicit presentations of the cohomology ring of the complement and the LCS Lie algebra of its fundamental group. Accordingly, we analyze these homological root homomorphisms in this section.

We continue with the notation of \S\ref{sec:toric}:
$\A$ is an essential toric arrangement in $(\C^\times)^{d+1}$, with $\P(\A_Y)$ a corank-one TM-ideal of $\P(\A)$ and $\A\smallsetminus\A_Y=\{H_1,\dots,H_l\}$. In coordinates $(x_1,\dots,x_d,y)=(\mathbf{x},y)$ on $\C^{d+1}$, the hypersurface $H_j$ is given by 
$y=\mu_j x_1^{m_{j,1}}x_2^{m_{j,2}} \cdots x_d^{m_{j,d}} = \mu_j\mathbf{x}^{\mathbf{m}_j}$, where $\mathbf{m}_j\in\Z^d$ and $\mu_j$ is some root of unity. Letting $\B$ denote the essential arrangement $\A/Y$ in $(\C^\times)^d$, the root map $\rootmap \colon M(\B) \to \Conf_n(\C)$, where $n=l+1$, is given by 
\begin{equation} \label{eq:rm}
\rootmap\colon \mathbf{x} \mapsto (b_1(\mathbf{x}),\dots,b_n(\mathbf{x})) = (0,\mu_1\mathbf{x}^{\mathbf{m}_1},\dots,\mu_l\mathbf{x}^{\mathbf{m}_l}).
\end{equation}

For a generic basepoint $\mathbf{x}_0$ in $M(\B)$, the components of $\rootmap(\mathbf{x}_0)$
have distinct real parts, which we use to order strands of (geometric) pure braids. If need be, by conjugating by a (non-pure) braid, we may ensure that the ordering of the roots in \eqref{eq:rm} corresponds to the ordering of the strands in the pure braid group $P_n$. (We suppress this conjugation from the notation.) In particular, the 
root $b_1(\mathbf{x})=0$ arising from the coordinate plane $y=0$ in $\C^{d+1}$ corresponds to the first/left-most strand in $P_n$. See \S\ref{subsec:typeC2} for an illustration. 

\subsection{Homology generators} \label{subsec:H1gens}
Assume without loss of generality that each hypersurface $H\in \B$ is connected, and is given, as in \eqref{eq:rou}, by the equation 
$x_1^{m_{H,1}}\cdots x_d^{m_{H,d}}=\mu_H$, where $\mu_H$ is a root of unity and $\mathbf{m}_H=(m_{H,1},\dots,m_{H,d})\in\Z^d$. We abbreviate this defining equation by $\mathbf{x}^{\mathbf{m_H}}=\mu_H$, and write $f_H(\mathbf{x})=\mathbf{x}^{\mathbf{m_H}}-\mu_H$.
The first integral (de Rham) cohomology group $H^1(M(\B))$ of the complement of $\B$ is free abelian  of rank $N=d+|\B|$. A basis is given by the set of differential forms 
\begin{equation}\label{eq:H^1basis}
\{\zeta_i \mid 1\le i \le d\}\cup\{\omega_H \mid H\in \B\},
\end{equation}
where $\zeta_i=\frac{1}{2\pi\ii}\mathrm{dlog}\, x_i$ and 
$\omega_H=\frac{1}{2\pi\ii}\mathrm{dlog}\, f_H=\frac{1}{2\pi\ii}\mathrm{dlog}(\mathbf{x}^{\mathbf{m}_H}-\mu_H)$,
see \cite{CDDMP} and the references therein (and recall that $\ii=\sqrt{-1}$).
We exhibit a dual basis for the first integral homology group $H_1(M(\B))$.

The complement $M(\B)$ may be realized as $\C^d\smallsetminus V$, where $V$ is the variety in $\C^d$ determined by the toric arrangement $\B$, together with the coordinate hyperplanes. Let $L=\{\mathbf{x}_0+z\cdot\mathbf{v}\mid z\in\C\}$, where $\mathbf{x}_0\in M(\B)$ is the basepoint noted above and $\mathbf{v}\in\C^d$, be a complex line in $\C^d$ that is transverse to $V$. The intersection $L\cap V$ is a collection of points: the $d$ points 
$\mathbf{q}_i$ where $L$ meets the coordinate hyperplane $\{x_i=0\}$, and, for each $H\in \B$, points $\mathbf{q}_{H,1},\dots \mathbf{q}_{H,r_H}$ where $L$ meets $H$.

For $\epsilon>0$ sufficiently small, each of the disks $D_\epsilon(\mathbf{q}_i)$, resp., $D_\epsilon(\mathbf{q}_{H,j})$, in $L$ of radius $\epsilon$ centered at the aforementioned points meets the variety $V$ at only $\mathbf{q}_i$, resp., $\mathbf{q}_{H,j}$. Let $\xi_i$ and $\xi_{H,j}$ denote the (oriented) boundary circles of the disks 
$D_\epsilon(\mathbf{q}_i)$ and $D_\epsilon(\mathbf{q}_{H,j})$,
\begin{equation}\label{eq:H1genrep}
\xi_{i}(t) = \mathbf{q}_i+\epsilon \exp(2 \pi \ii t)\cdot\mathbf{v},\quad \xi_{H,j}(t) = \mathbf{q}_{H,j}+\epsilon \exp(2 \pi \ii t)\cdot\mathbf{v},\quad 0\le t \le 1.
\end{equation}
Viewing these circles in $M(\B)$ via the inclusion $L\cap M(\B) \hookrightarrow M(\B)$, denote their homology classes
in $H_1(M(\B))$ as $X_i=[\xi_i]$ and $X_{H,j}=[\xi_{H,j}]$. As we will see below, for $H\in \B$, the classes $X_{H,j}$ and $X_{H,k}$ coincide for $1\le  j,k\le r_H$. So write $X_H=X_{H,1}$ for $H\in\B$.

\begin{proposition} \label{prop:H1basis}
The first homology group $H_1(M(\B)) \cong \Z^{d+|\B|}$ has basis 
\begin{equation} \label{eq:H_1basis}
\{X_i \mid 1\le i \le d\} \cup \{X_H \mid H \in \B\},
\end{equation}
dual to the basis \eqref{eq:H^1basis} of $H^1(M(\B))$.
\end{proposition}
\begin{proof}
We use the (perfect) Kronecker pairing $H^1(M(\B))\otimes H_1(M(\B)) \to \Z$, given, in terms of (the cohomology class of) a differential form $\omega$ and a homology class $X=[\xi]$, by integration, $\langle\omega,X\rangle=\int_\xi \omega$.

Restricting the functions $f_H$ giving the hypersurfaces in $\B$ 
to the line $L=\{\mathbf{x}_0+z\cdot\mathbf{v}\}$ yields
functions of one complex variable
$\phi_{H}(z)=f_H(\mathbf{x}_0+z\cdot\mathbf{v})$ and $\psi_{H}(z)=\phi'_{H}(z)/\phi_{H}(z)$. Similarly,  restricting the coordinate functions $x_i$ yields $\phi_i(z)=x_i(\mathbf{x}_0+z\cdot\mathbf{v})$ and $\psi_i(z)=\phi'_i(z)/\phi_i(z)$. Express the points where $L$ meets $\{x_i=0\}$
as $\mathbf{q}_i=\mathbf{x}_0+z_i\cdot \mathbf{v}$, and the points where $L$ meets $H\in\B$ as $\mathbf{q}_{H,j}=\mathbf{x}_0+z_{H,j}\cdot \mathbf{v}$, $1\le j\le r_H$. 
Note that $\phi_i(z_i)=0$, that $\phi_H(z_{H,j})=0$ for each $j$, $1\le j\le r_H$,  and that these are the only zeros of these functions.
Writing $\mathbf{x}_0=(x_{0,1},\dots,x_{0,d})$ and $\mathbf{v}=(v_{1},\dots,v_{d})$, it is immediate that the function $\psi_i(z)=v_i/(x_{0,i}+v_iz)$ has a simple pole at $z_i=-x_{0,i}/v_i$, with residue $1$.
Since $\lim_{z\to z_{H,j}} (z-z_{H,j})\cdot \psi_{H}(z)=1$ by L'Hospital's rule, the function $\psi_{H}(z)$ has a simple pole at $z_{H,j}$, with residue $1$, for each $j$, $1\le j\le r_H$.

With these observations, one can use the Residue Theorem to show that the set \eqref{eq:H_1basis} is a basis for $H_1(M(\B))$, dual to the basis \eqref{eq:H^1basis} for $H^1(M(\B))$. First consider the Kronecker pairing 
$\langle \omega_K,X_{H,j}\rangle$ for $H,K\in\B$. We have
\[
\langle \omega_K,X_{H,j}\rangle=\int_{\xi_{H,j}} \omega_K=\int_{\xi_{H,j}}{\textstyle\frac{1}{2\pi\ii}}\mathrm{dlog} f_K(\mathbf{x})={\textstyle\frac{1}{2\pi\ii}}\int_{\gamma_{H,j}} \psi_K(z),
\]
where $\gamma_{H,j}$ is the circle of radius $\epsilon$ centered at $z_{H,j}$, oriented counterclockwise. If $K\neq H$, then from our choice of $\epsilon$, none of the poles $z_{K,l}$ of $\psi_K(z)$ lie within this circle and 
\[\langle \omega_K,X_{H,j}\rangle=\frac{1}{2\pi\ii}\int_{\gamma_{H,j}} \psi_K(z)=0.\] Similarly, we have  
$\langle \zeta_i,X_{H,j}\rangle=\frac{1}{2\pi\ii}\int_{\gamma_{H,j}} \psi_i(z)=0$, since the pole $z_i$ of $\psi_i(z)$ does not lie within this circle either.

If $K=H$, then
\[
\langle \omega_H,X_{H,j}\rangle={\textstyle\frac{1}{2\pi\ii}}\int_{\gamma_{H,j}} \psi_H(z)=1
\]
by the Residue Theorem, since $z_{H,j}$ is the only pole of $\psi_H(z)$ within the circle $\gamma_{H,j}$.
Thus, for $1\le j,k\le r_H$, the homology classes $X_H=X_{H,j}$ and $X_{H,k}$ are two expressions of the dual of the cohomology class $\omega_H$, hence are equal. In particular, $X_H=X_{H,1}$ is the dual of $\omega_H$.

Similar considerations yield
\[
\langle \zeta_i,X_{j}\rangle={\textstyle\frac{1}{2\pi\ii}}\int_{\gamma_{j}} \psi_i(z)=\delta_{i,j}
\quad\text{and}\quad
\langle \zeta_i,X_H\rangle=\langle \zeta_i,X_{H,1}\rangle={\textstyle\frac{1}{2\pi\ii}}\int_{\gamma_{H,1}} \psi_i(z)=0.
\]
If $i=j$, a quick calculation yields $\langle \zeta_i,X_{j}\rangle=1$. The remaining cases follow from the fact that 
the pole $z_i$ of $\psi_i(z)$ does not lie within either the circle $\gamma_{H,1}$ or, for $i\neq j$, the circle $\gamma_j$ of radius $\epsilon$ centered at $z_j$. Consequently, the homology class $X_i$ is the dual of the cohomology class $\zeta_i$
\end{proof}

\begin{remark}\label{rem:confgen}
We will also make use 
of explicit generators for the first integral homology group of the configuration space $\Conf_n(\C)$. Recall from \Cref{ex:lcsPn} that $H_1(\Conf_n(\C))=H_1(P_n)$ (the first graded piece of the LCS Lie algebra $\lie(P_n)$) has basis $\{A_{i,j}\mid 1\le i<j\le n\}$. The classes $A_{i,j}$ may be represented by loops in $\Conf_n(\C)$ about the diagonal hyperplanes $\Delta_{i,j}=\{x_i=x_j\}$, and are dual to the classical generators of the integral cohomology ring $H^*(\Conf_n(\C))$ which we recall next. For $1\le i<j\le n$, define $\proj_{i,j}\colon \Conf_n(\C) \to \C^\times$ by $\proj_{i,j}(x_1,\dots,x_n)=x_j-x_i$. 
Fixing a generator $\omega=\frac{1}{2\pi\ii} \mathrm{dlog}\,z$ of $H^1(\C^\times)$ yields classes $\omega_{i,j}=\proj_{i,j}^*(\omega)=\frac{1}{2\pi\ii} \mathrm{dlog}(x_j-x_i)$ in $H^1(\Conf_n(\C))$, which generate the ring $H^*(\Conf_n(\C))$. Then $\omega_{i,j}\in H^1(\Conf_n(\C))$ and $A_{i,j}\in H_1(\Conf_n(\C))$ are dual.
\end{remark}

\subsection{The homological root homomorphism} \label{subsec:hrh}
We now determine the map 
\[
\rootmap_*\colon H_1(M(\B))\to H_1(\Conf_n(\C))
\]
in homology 
induced by the root map \eqref{eq:rm}. Recall that this map is given explicitly by $\rootmap(\mathbf{x})=(b_1(\mathbf{x}),b_2(\mathbf{x}),\dots,b_n(\mathbf{x}))$, where $b_1(\mathbf{x})=0$, and $b_{j+1}(\mathbf{x})=\mu_{j}\mathbf{x}^{\mathbf{m}_{j}}=\mu_jx_1^{m_{j,1}}\cdots x_d^{m_{j,d}}$ for $1\le j\le l$, where $l=|\A|-|\B|$ and $n=l+1$.  
Our goal is to find an explicit description of $\rootmap_\ast$ in terms of the generators of $H_1(M(\B))$ and $H_1(\Conf_n(\C))$ given in \Cref{prop:H1basis} and \Cref{rem:confgen}.

\begin{remark} \label{rem:comps}
Let $\pi\colon (\C^\times)^{d+1} \to (\C^\times)^d$, $\pi(\mathbf{x},y)=\mathbf{x}$, be the projection map which forgets the last coordinate. For $H_i,H_j \in \A\smallsetminus \A_Y$ and any connected component $L$ of $H_i\cap H_j$, \Cref{lem:layermap} implies that $\pi(L)$ is a layer of $\B$ of dimension equal to the dimension of $L$. Therefore $\pi(L)$ is a hypersurface of $\B$.

Explicitly, if $H_i$ and $H_j$ are given by $y=\mu_i\mathbf{x}^{\mathbf{m}_i}$ and $y=\mu_j\mathbf{x}^{\mathbf{m}_j}$, let $\mathbf{a}=(\mathbf{m}_j-\mathbf{m}_i)/m(i,j)$, where
$m(i,j)=\gcd(m_{j,1}-m_{i,1},\dots,m_{j,d}-m_{i,d})$.
The intersection $H_i\cap H_j$ is then given by the equation $(\mathbf{x}^{\mathbf{a}})^{m(i,j)}-\mu_i\mu_j^{-1}=0$, 
with connected components given by $\mathbf{x}^{\mathbf{a}}-\lambda\eta^q=0$, where $\lambda=|\mu_i\mu_j^{-1}|^{1/m(i,j)}$,  $\eta=\exp(2\pi\ii/m(i,j))$, and $1\le q \le m(i,j)$. Thus, $\mathbf{x}^{\mathbf{a}}-\lambda\eta^q=0$ is an equation
of a hypersurface $H$ of $\B$. Since $H$ is given by the vanishing of $f_H(\mathbf{x})=\mathbf{x}^{\mathbf{m}_H}-\mu_H$, we have either $f_H(\mathbf{x})=\mathbf{x}^{\mathbf{a}}-\lambda\eta^q$ (when $\mathbf{m}_H=\mathbf{a}$), or $f_H(\mathbf{x})=\mathbf{x}^{\mathbf{-a}}-\lambda^{-1}\eta^{-q}$ (when $\mathbf{m}_H=-\mathbf{a}$). If $s(i,j)=\bigl|\{H\in\B \mid \mathbf{m}_H=\mathbf{a}\}\bigr|$, then 
$m(i,j)-s(i,j)=\bigl|\{H\in\B \mid \mathbf{m}_H=-\mathbf{a}\}\bigr|$, and
\[
{\boldsymbol{\nu}}(i,j):=
\frac{1}{m(i,j)}\bigl[s(i,j)\mathbf{m}_i+(m(i,j)-s(i,j))\mathbf{m}_j\bigr]=\mathbf{m}_i+(m(i,j)-s(i,j))\mathbf{a} =\mathbf{m}_j-s(i,j)\mathbf{a}\in \Z^d
\]
is an integer vector. Write ${\boldsymbol{\nu}}(i,j)=\bigl(\nu_1(i,j),\dots,\nu_d(i,j)\bigr)$.
\end{remark}

\begin{theorem} \label{thm:hrm}
Let $\A$ be an essential toric arrangement in $(\C^\times)^{d+1}$, with $\P(\A_Y)$ a corank-one TM-ideal of $\P(\A)$, $\A\smallsetminus\A_Y=\{H_1,\dots,H_l\}$, and $\B=\A/Y$. 
On the basis $\{X_k \mid 1\le k \le d\} \cup \{X_H \mid H \in \B\}$ of $H_1(M(\B))$,  
the homological root homomorphism $\rootmap_*\colon H_1(M(\B))\to H_1(\Conf_n(\C))$,
where $n=l+1$, is given by
\[
\rootmap_*(X_{k})=\sum_{j=1}^{n-1}\Bigl[m_{j,k}A_{1,j+1} + \sum_{i=1}^{j-1} \nu_k(i,j) A_{i+1,j+1}\Bigr]  \quad  \text{and} \quad   
\rootmap_*(X_{H})=\sum
A_{i+1,j+1},  
\]
the latter sum running over all $i<j$ for which $H$ is a connected component of $\pi(H_i \cap H_j)$.
\end{theorem}
\begin{proof}
We again use the Kronecker pairing. For $X\in H_1(M(\B))$, the coefficient of $A_{r,s}$ in $\rootmap_*(X)$ is equal to $\langle\omega_{r,s},\rootmap_*(X)\rangle=\langle\rootmap^*(\omega_{r,s}),X\rangle$.

We determine the map $\rootmap^*\colon H^1(\Conf_n(\C))\to H^1(M(\B))$ in cohomology.  
If $r=1$ and $s=j+1$, then
\begin{equation} \label{eq:w1j+1}
\rootmap^*(\omega_{1,j+1})={\textstyle\frac{1}{2\pi\ii}}\mathrm{dlog}(b_{j+1}(\mathbf{x})-b_1(\mathbf{x}))=
{\textstyle\frac{1}{2\pi\ii}}\mathrm{dlog}(\mu_j\mathbf{x}^{\mathbf{m}_j})=
\sum_{k=1}^d m_{j,k}\zeta_k.
\end{equation}
If $r=i+1$ and $s=j+1$ for $1\le i<j\le l$, then $b_{j+1}(\mathbf{x})-b_{i+1}(\mathbf{x})=
\mu_j\mathbf{x}^{\mathbf{m}_j}-\mu_i\mathbf{x}^{\mathbf{m}_i}$. 
In the notation established in \Cref{rem:comps}, we have
\[
\mu_j\mathbf{x}^{\mathbf{m}_j}-\mu_i\mathbf{x}^{\mathbf{m}_i} 
=\mu_j \mathbf{x}^{\mathbf{m}_i} (\mathbf{x}^{\mathbf{m}_j-\mathbf{m}_i} - \mu_i\mu_j^{-1})=
\mu_j \mathbf{x}^{\mathbf{m}_i} ((\mathbf{x}^{\mathbf{a}})^{m(i,j)} - \mu_i\mu_j^{-1})=
\mu_j \mathbf{x}^{\mathbf{m}_i} \prod_{q=1}^{m(i,j)} (\mathbf{x}^{\mathbf{a}}-\lambda\eta^q),
\]
and the vanishing of each factor $\mathbf{x}^{\mathbf{a}}-\lambda\eta^q$ gives a hypersuface $H$ of $\B$, also defined by the vanishing of $f_H(\mathbf{x})=\mathbf{x}^{\mathbf{m}_H}-\mu_H$. In the $s(i,j)$ instances where $\mathbf{m}_H=\mathbf{a}$, we have $\mathrm{dlog}(\mathbf{x}^{\mathbf{a}}-\lambda\eta^q)=\mathrm{dlog}\,f_H$, while in the $m(i,j)-s(i,j)$ instances where $\mathbf{m}_H=-\mathbf{a}$, we have 
$\mathrm{dlog}(\mathbf{x}^{\mathbf{a}}-\lambda\eta^q)=\mathrm{dlog}\,f_H+\mathrm{dlog}(\mathbf{x}^{\mathbf{a}})$. These observations yield
\[
\begin{aligned}
\rootmap^*(\omega_{i+1,j+1})&={\textstyle\frac{1}{2\pi\ii}}\mathrm{dlog}(b_{j+1}(\mathbf{x})-b_{i+1}(\mathbf{x}))={\textstyle\frac{1}{2\pi\ii}}\mathrm{dlog}(\mu_j\mathbf{x}^{\mathbf{m}_j}-\mu_i\mathbf{x}^{\mathbf{m}_i})\\
&={\textstyle\frac{1}{2\pi\ii}}\mathrm{dlog}(\mathbf{x}^{\mathbf{m}_i})
+\sum_{q=1}^{m(i,j)}{\textstyle\frac{1}{2\pi\ii}}\mathrm{dlog}(\mathbf{x}^{\mathbf{a}}-\lambda\eta^q)\\
&={\textstyle\frac{1}{2\pi\ii}}\mathrm{dlog}(\mathbf{x}^{\mathbf{m}_i})
+\sum_{\mathbf{m}_H=\mathbf{a}}{\textstyle\frac{1}{2\pi\ii}}\mathrm{dlog}\,f_H
+\sum_{\mathbf{m}_H=-\mathbf{a}}{\textstyle\frac{1}{2\pi\ii}}(\mathrm{dlog}\,f_H+\mathrm{dlog}(\mathbf{x}^{\mathbf{a}}))\\
&={\textstyle\frac{1}{2\pi\ii}}\mathrm{dlog}(\mathbf{x}^{\mathbf{m}_i})
+{\textstyle\frac{1}{2\pi\ii}}(m(i,j)-s(i,j)){\textstyle\frac{1}{2\pi\ii}}\mathrm{dlog}(\mathbf{x}^{\mathbf{a}})+ 
\sum_{\genfrac{}{}{0pt}{}{H\in\B}{H\subset \pi(H_i\cap H_j)}}{\textstyle\frac{1}{2\pi\ii}}\mathrm{dlog}\,f_H.
\end{aligned}
\]
Since ${\boldsymbol{\nu}(i,j)}=\mathbf{m}_i+(m(i,j)-s(i,j))\mathbf{a}$, we have
$\mathrm{dlog} ( \mathbf{x}^{\boldsymbol{\nu}(i,j)})= \mathrm{dlog}(\mathbf{x}^{\mathbf{m}_i})+
(m(i,j)-s(i,j))\mathrm{dlog}(\mathbf{x}^{\mathbf{a}})$. Consequently, 
\begin{equation} \label{eq:wi+1j+1}
\rootmap^*(\omega_{i+1,j+1})=
{\textstyle\frac{1}{2\pi\ii}}\mathrm{dlog}(\mathbf{x}^{\boldsymbol{\nu}(i,j)})+
\sum_{\genfrac{}{}{0pt}{}{H\in\B}{H\subset \pi(H_i\cap H_j)}}{\textstyle\frac{1}{2\pi\ii}}\mathrm{dlog}\,f_H
=\sum_{k=1}^d \nu_k(i,j)\zeta_k + 
\sum_{\genfrac{}{}{0pt}{}{H\in\B}{H\subset \pi(H_i\cap H_j)}} \omega_H 
\end{equation}

It is now a straightforward exercise using \Cref{prop:H1basis},  
\eqref{eq:w1j+1}, and \eqref{eq:wi+1j+1} to determine 
the homological root homomorphism $\rootmap_*\colon H_1(M(\B))\to H_1(\Conf_n(\C))$. 
For $X\in H_1(M(\B))$, write $\rootmap_*(X)=\sum_{1\le r<s\le n} c_{r,s}A_{r,s}$,  where $c_{r,s}=\langle\omega_{r,s},\rootmap_*(X)\rangle=\langle\rootmap^*(\omega_{r,s}),X\rangle$.

First, consider the generator $X=X_H$ corresponding to $H\in\B$. If $r=1$ and $s=j+1$, then 
$c_{1,j+1}=\sum_{k=1}^d m_{j,k} \langle \zeta_k,X_H\rangle=0$ using \eqref{eq:w1j+1}.
If $r=i+1$ and $s=j+1$, using \eqref{eq:wi+1j+1}, we have
\[
c_{i+1,j+1}=
\sum_{k=1}^d m_{i,k}\langle\zeta_k,X_H\rangle + 
\sum_{\genfrac{}{}{0pt}{}{K\in\B}{K\subset \pi(H_i\cap H_j)}} \langle\omega_K,X_H
\rangle=
\sum_{\genfrac{}{}{0pt}{}{K\in\B}{K\subset \pi(H_i\cap H_j)}} \langle\omega_K,X_H\rangle
=\delta_{K,H}.
\]
Thus, $\rootmap_*(X_H)=\sum A_{i+1,j+1}$, where the sum is over $i<j$ for which $H$ is a connected component of $\pi(H_i \cap H_j)$ as asserted.

Next, consider the generator $X=X_k$ corresponding to the coordinate hyperplane $\{x_k=0\}$.
If $r=1$ and $s=j+1$, then 
$c_{1,j+1}=\sum_{q=1}^d m_{j,q} \langle \zeta_q,X_k\rangle=m_{j,k}$ using \eqref{eq:w1j+1}.
If $r=i+1$ and $s=j+1$, using \eqref{eq:wi+1j+1}, we have
\[
c_{i+1,j+1}=
\sum_{q=1}^d \nu_q(i,j)\langle\zeta_q,X_k\rangle + 
\sum_{\genfrac{}{}{0pt}{}{H\in\B}{H\subset \pi(H_i\cap H_j)}} \langle\omega_H ,X_k
\rangle=\nu_k(i,j).
\]
This completes the proof.
\end{proof}

\begin{example} 
Let $\CC$ be the rank two circuit in $(\C^\times)^2$ considered in \Cref{subsec:rank2circ}, defined by $x-\mu^j=0$, $1\le j \le n$, $x^{m_1}y=1$, and $x^{m_2}y=1$, where $m=m_2-m_1>0$, $n=km$ and $\mu=\exp(2\pi\ii /n)$. With $H_1$ given by $y=x^{-m_1}$ and $H_2$ by $y=x^{-m_2}$, the point $x=\mu^j$ is a component of $\pi(H_1\cap H_2)$ when $j=qk$, so that $\mu^j$ is an $m$-th root of unity. 
As in \S\ref{subsec:r2lcs}, denote the generators of the first homology of $M(\B)=\C\smallsetminus\{0,1,\mu,\dots,\mu^{n-1}\}$ by $X_j$, $0\le j \le n$, where $X_0$ is the class of a loop about $0$, and $X_j$ is the class of a loop about $\mu^j$ for $j \ge 1$. 
By \Cref{thm:hrm}, the homological root homomorphism associated to the root map $\rootmap(x)=(0,x^{-m_1},x^{-m_2})$ is given by
\[
\rootmap_*(X_0)=-m_1 A_{1,2}-m_2A_{1,3}-m_2A_{2,3}, \qquad \rootmap_*(X_j)=\begin{cases}A_{2,3}&\text{if $j=qk$,} \\
0 &\text{otherwise}.
\end{cases}
\]
Noting that 
\[
\rootmap_*(X_0)=(m_2-m_1) A_{1,2}-m_2(A_{1,2}+A_{1,3}+A_{2,3})=m A_{1,2}-m_2(A_{1,2}+A_{1,3}+A_{2,3}),
\]
it is readily checked that the resulting LCS Lie algebra $\lie(\CC)$ and cohomology ring $H^*(M(\CC))$ are isomorphic to $\lie(\CC_{n,m})$ and $H^*(M(\CC_{n,m}))$, obtained from 
the homological root homomorphism for the ``standard form'' rank two circuit $\CC_{n,m}$ recorded in \S\ref{subsec:r2lcs} (and easily recoverable from \Cref{thm:hrm}).
\end{example}

\section{Type C toric arrangements} \label{sec:typeC}
Let $\CC_n$ denote the Weyl type C toric arrangement in $(\C^\times)^n$, consisting of the hypersurfaces $x_i^2-1=0$ ($1\le i \le n)$, $x_i^{-1}x_j-1=0$, $x_ix_j-1=0$ ($1\le i<j\le n$), with complement $M(\CC_n)$. As the poset of layers $\P(\CC_n)$ is a Dowling poset, the arrangement $\CC_n$ is strictly supersolvable, see \cite[Example 5.1.8]{BD}. Topologically, repeatedly forgetting the last coordinate gives rise to a tower of fiber bundles
\[
M(\CC_n) \to M(\CC_{n-1}) \to \dots \to M(\CC_2) \to M(\CC_1).
\]
We illustrate our results by determining the integral lower central series Lie algebra and cohomology ring of the fundamental group of $M(\CC_n)$.

\subsection{The case \texorpdfstring{$n=2$}{n=2}} \label{subsec:typeC2} We begin in rank two, where we record the almost-direct product structure of the fundamental group of the complement of the type C toric arrangement in $(\C^\times)^2$.

Consider the type C toric arrangements $\CC_1$ in $\C^\times\subset \C$ and $\CC_2$ in $(\C^\times)^2 \subset \C^2$, given by the vanishing of the polynomials $x(x^2-1)$ and $x(x^2-1)y(y^2-1)(y-x)(xy-1)$, respectively. The strictly supersolvable bundle $M(\CC_2)\to M(\CC_1)$ is equivalent to the pullback of the Fadell-Neuwirth bundle $\Conf_6(\C) \to \Conf_5(\C)$ along the map 
$\rootmap\colon M(\CC_1) \to \Conf_5(\C)$ given by $\rootmap(x)=(0,1,-1,x,x^{-1})$. 

In $M(\CC_1)=\C\smallsetminus\{-1,0,1\}$, define loops $\gamma_1(t)=1-\frac{1}{2}\exp(2 \pi \ii t)$, $\gamma_{-1}(t)=-1+\frac{1}{2}\exp(2 \pi \ii t)$, $0\le t \le 1$, and paths $\gamma_0^+(t)=\frac{1}{2}\exp(2 \pi \ii t)$, $\gamma_0^-(t)=\frac{1}{2}\exp(2 \pi \ii (t+\frac{1}{2}))$, $0 \le t \le \frac{1}{2}$, where $\ii=\sqrt{-1}$, see \Cref{fig:CC1}. The fundamental group $\pi_1(M(\CC_1))$, based at $*=x_0=\frac{1}{2}$, is generated by the homotopy classes of the loops 
$\gamma_0^+\centerdot \gamma_{-1} \centerdot \bar{\gamma}_0^+$, $\gamma_0^+\centerdot \gamma_0^-$, and $\gamma_1$, where $\bar\lambda(t)=\lambda(1-t)$ is the reverse path and $\lambda\centerdot\mu$ is concatenation. Denoting these classes by $\ns_{1}$ (loop about $x=-1$), $\zs_1$ (about $x=0$), and $\ps_1$ (about $x=1$), $\pi_1(M(\CC_1),x_0)$ is the free group $F_3=\langle\zs_1,\ps_1,\ns_1\rangle$ on three generators. 
\begin{figure}[h] 
\begin{tikzpicture} 
\draw [-stealth](1.5,0) -- (1.5,.01);
\draw [-stealth](-1.5,0) -- (-1.5,-.01);
\draw [-stealth](0,.5) -- (-.01,.5);
\draw [-stealth](0,-.5) -- (.01,-.5);
\draw[color=black] (1,0) circle [radius=.5]; 
\draw[color=black] (0,0) circle [radius=.5]; 
\draw[color=black] (-1,0) circle [radius=.5]; 
\node at (0,0) {\small{$0$}};
\node[color=black] at (1,0) {\small{$1$}};
\node[color=black] at (-1.1,0) {\small{$-1$}};
\node[color=black] at (1.8,0) {\small{$\gamma_1$}};
\node[color=black] at (-1.8,0) {\small{$\gamma_{-1}$}};
\node[color=black] at (0,.8) {\small{$\gamma^+_0$}};
\node[color=black] at (0,-.8) {\small{$\gamma^-_0$}};
\node[color=black,anchor=center,inner sep=.5,fill=white] at (.5,0) {$*$};
\end{tikzpicture} 
\caption{Loops and paths in $M(\CC_1)$} \label{fig:CC1} 
\end{figure}

The above paths may be used to determine the map $\rootmap_\sharp \colon F_3 \to P_5$ induced by $\rootmap$ on fundamental groups, where $P_5=\pi_1(\Conf_5(\C),\rootmap(x_0))$ is the 5-string pure braid group. Ordering braid strands by increasing real part at the basepoint $\rootmap(x_0)=(0,1,-1,\frac{1}{2},2)$, one can check that
\[
\zs_{1}\mapsto a_{2,3}\bigl(a_{1,5}a_{2,5}a_{3,5}a_{4,5}\bigr)^{-1}, \quad \ps_{1}\mapsto a_{3,4}a_{3,5}a_{4,5},
\quad \ns_{1}\mapsto\bigl(a_{1,3}a_{1,5}a_{3,5}\bigr)^{a_{2,5}a_{3,5}a_{4,5}}, 
\]
where the $a_{i,j}$ are the standard generators of the pure braid group and $u^v=v^{-1}uv$. 
Conjugating by $\sigma_1^{-1}\sigma_3^{}\sigma_2^{}\in B_5$ yields an automorphism of $P_5$, $u\mapsto u^{\sigma_1^{-1}\sigma_3^{}\sigma_2^{}}$, which insures that the strands of $P_5$ correspond to the order of the roots given by $\rootmap(x)$ as in \Cref{sec:hrm} (and slightly simplifies subsequent fundamental group calculations). Carrying this out yields
\begin{equation} \label{eq:pi1root}
\rootmap_\sharp(\zs_{1})=a_{1,4}^{a_{2,4}a_{3,4}}a_{4,5}^{-1}a_{3,5}^{-1}a_{2,5}^{-1}a_{1,5}^{-1}, \quad  \rootmap_\sharp(\ps_{1})=a_{2,4}a_{2,5}a_{4,5}, \quad 
\rootmap_\sharp(\ns_{1})=\bigl(a_{3,4}a_{3,5}a_{4,5}\bigr)^{a_{1,4}^{-1}a_{1,3}^{-1}a_{1,2}^{-1}}.  
\end{equation}

The map $\rootmap_\sharp \colon F_3 \to P_5$, together with 
the Artin representation $\hat\alpha_5\colon P_5 \to \Aut(F_5)$, determines the almost-direct product structure of the fundamental group $\pi_1(M(\CC_2)) = F_5 \rtimes_\phi F_3$, where $\phi=\hat\alpha_5 \circ \rootmap_\sharp$. Denote the generators of $F_5$ by
$\ys_1=\zs_2$, $\ys_2=\ps_2$, $\ys_3=\ns_2$, $\ys_4=\as_{1,2}$, and $\ys_5=\bs_{1,2}$,
when viewing them as elements of $\pi_1(M(\CC_2))$. Then, the group $\pi_1(M(\CC_2))$ has generators $\zs_1,\ps_1,\ns_1,\zs_2,\ps_2,\ns_2,\as_{1,2},\bs_{1,2}$, and relations $u^{-1}v u = \phi(u)(v)$, where $\phi(u)(v)=w(u,v) \cdot v\cdot w(u,v)^{-1}=[w(u,v),v]\cdot v$ for $u$ and $v$ generators of $F_3$ and $F_5$, respectively. Letting $\vs=\zs_2^{\ps_2\ns_2\as_{1,2}}$ and $\ws=\zs_2^{}\ps_2^{}\ns_2^{}\as_{1,2}^{}\bs_{1,2}^{}$, calculations with the Artin representation \eqref{eq:PureArtin} yield: 
\[
\begin{matrix}
\begin{aligned}
    w(\zs_1,\zs_2)&=\bs_{1,2}^{-1}\ws\bs_{1,2}^{-1}\ns_2^{-1}\ps_2^{-1}, \\
    w(\zs_1,\ps_2)&=\bs_{1,2}^{-1}, \\
    w(\zs_1,\ns_2)&= \bs_{1,2}^{-1},\\
    w(\zs_1,\as_{1,2})&= \bs_{1,2}^{-1}\as_{1,2}^{}\vs,\\
    w(\zs_1,\bs_{1,2})&=\ws^{-1},
\end{aligned}  
&
\begin{aligned}
    w(\ps_1,\zs_2)&=1, \\
    w(\ps_1,\ps_2)&=\ps_2\as_{1,2}^{}\bs_{1,2}^{}, \\
    w(\ps_1,\ns_2)&= [\ps_2^{},\as_{1,2}^{}\bs_{1,2}^{}],\\
    w(\ps_1,\as_{1,2})&=\ps_2\as_{1,2}^{}\bs_{1,2}^{}, \\
    w(\ps_1,\bs_{1,2})&=\ps_2\as_{1,2}^{}\bs_{1,2}^{},
\end{aligned}  
&
\begin{aligned}
    w(\ns_1,\zs_2)&=\ws\vs^{-1}\ps_2^{-1} \\
    w(\ns_1,\ps_2)&= 1,\\
    w(\ns_1,\ns_2)&=\ps_2^{-1}\ws\vs^{-1}, \\
    w(\ns_1,\as_{1,2})&=\ps_2^{-1}\ws\vs^{}, \\
    w(\ns_1,\bs_{1,2})&=\vs^{-1}\ps_2^{-1}\ws.
\end{aligned}  
\end{matrix}
\]

Passing to homology, denote the generators of $H_1(M(\CC_1))\cong\Z^3$ and $H_1(\Conf_5(\C))\cong \Z^{10}$ by $Z_1=[\zs_1]$, $P_1=[\ps_1]$, $N_1=[\ns_1]$, and $A_{i,j}=[a_{i,j}]$, $1\le i<j\le 5$. From \eqref{eq:pi1root} or \Cref{thm:hrm}, the  
homological root homomorphism $\rootmap_*\colon H_1(M(\CC_1)) \to H_1(\Conf_5(\C))$ is then given by
\[
\rootmap_*(Z_1)=A_{1,4}-A_{1,5}-A_{2,5}-A_{3,5}-A_{4,5}, \quad \rootmap_*(P_1)=A_{2,4}+A_{2,5}+A_{4,5}, \quad 
\rootmap_*(N_1)=A_{3,4}+A_{3,5}+A_{4,5}.  
\]
By Theorems \ref{thm:LCSliealg} and \ref{thm:H*ring}, this may be used to determine the LCS Lie algebra $\lie(\CC_2)$ and the cohomology ring $H^*(M(\CC_2))$. We discuss the requisite calculations for general $n$ below.

\subsection{General \texorpdfstring{$n$}{n}} \label{subsec:typeCn}
Let $\CC_n$ denote the type C toric arrangement in $(\C^\times)^n \subset\C^n$, given by the (connected) hypersurfaces
\[
x_i^{}=0,\quad  x_i^{}=1,\quad x_i^{}=-1 \quad (1\le i \le n),\qquad x_j^{}=x_i^{},\quad x_i^{}x_j^{}=1 \quad (1\le i<j\le n).
\]
Let $M(\CC_n)$ be the complement, with fundamental group $G(\CC_n)=\pi_1(M(\CC_n))$. Since $M(\CC_n)$ is a $K(G(\CC_n),1)$-space, we have $H_*(M(\CC_n))=H_*(G(\CC_n))$, $H^*(M(\CC_n))=H^*(G(\CC_n))$, etc. 

The first integral cohomology $H^1(M(\CC_n))$ is free abelian of rank $n+2n+2\binom{n}{2}$, generated by 
\begin{equation} \label{eq:Cforms}
\begin{array}{lll}
\ \hskip 5pt z_i= \frac{1}{2\pi\ii}\mathrm{dlog}(x_i),&
\ \hskip 3.5pt \rho_i= \frac{1}{2\pi\ii}\mathrm{dlog}(x_i-1),&
\eta_i= \frac{1}{2\pi\ii}\mathrm{dlog}(x_i+1),\\
\alpha_{i,j}= \frac{1}{2\pi\ii}\mathrm{dlog}(x_j-x_i),\quad&
\beta_{i,j}= \frac{1}{2\pi\ii}\mathrm{dlog}(x_ix_j-1).\quad
\end{array}
\end{equation}
Denote the dual basis for 
$H_1(M(\CC_n))$ by $\{Z_i, P_i, N_i \mid 1\le i\le n\}\cup\{ 
X_{i,j}, Y_{i,j} \mid 1\le i<j\le n\}$, with $Z_i$ dual to $z_i$, $P_i$ dual to $\rho_i$, $N_i$ dual to $\eta_i$, $X_{i,j}$ dual to $\alpha_{i,j}$, and $Y_{i,j}$ dual to $\beta_{i,j}$ (so that, for instance, $Y_{i,j}$ denotes the class of a loop in $M(\CC_n)$ about only the hypersurface $x_ix_j=1$).

Let $\lie_n=\lie(\CC_n)$ be the integral LCS Lie algebra of the group $G(\CC_n)$, and let $\clie_n=\LL[H_1(M(\CC_n))]$ be the free Lie algebra generated by $Z_i,P_i,N_i$ ($1\le i \le n$), $X_{i,j},Y_{i,j}$ ($1\le i<j\le n$).

\begin{theorem} \label{thm:Cholo} The Lie algebra $\lie_n \cong \clie_n/{\mathcal J}_n$ is isomorphic to the quotient of the free Lie algebra $\clie_n$ by the Lie ideal ${\mathcal J}_n$ generated, for $1\le i,j,k\le n$ 
with $i<j$, resp., $i<j<k$, where relevant,  by
\[
\begin{array}{ll}
{\displaystyle{{[}Z_i-Z_j-P_j-N_j-\sum_{q=1}^{j-1} (X_{q,j}+ Y_{q,j}),\,Y_{i,j}],}}\\
{[}Z_i+Z_j+X_{i,j}-Y_{i,j},\,X]\  \text{for}\ X=Z_j, X_{i,j}, & {[}Z_i-Y_{i,j},\,X]\   \text{for}\ X=P_j, N_j, X_{q,j}, Y_{q,j}\, (q\neq i), \\{[}P_{i}+P_j+X_{i,j}+Y_{i,j},\,X]\  \text{for}\ X=P_j, X_{i,j}, Y_{i,j},
&
{[}P_{i},\,X]\  \text{for}\ X=Z_j, N_j, X_{q,j}, Y_{q,j}\, (q\neq i),\\  
{[}N_{i}+N_j+X_{i,j}+Y_{i,j},\,X]\  \text{for}\ X=N_j,X_{i,j}, Y_{i,j},
&
{[}N_{i},\,X]\  \text{for}\ X=Z_j, P_j, X_{q,j}, Y_{q,j}\, (q\neq i),\\
{[}X_{i,j}+X_{i,k}+X_{j,k},\,X]\  \text{for}\ X=X_{i,k},X_{j,k},
&
{[}Y_{i,j}+X_{i,k}+Y_{j,k},\,X]\  \text{for}\ X=X_{i,k},Y_{j,k},
\\
{[}X_{i,j}+Y_{i,k}+Y_{j,k},\,X]\  \text{for}\ X=Y_{i,k},Y_{j,k},&
{[}Y_{i,j}+Y_{i,k}+X_{j,k},\,X]\  \text{for}\ X=Y_{i,k},X_{j,k},\\
{[}X_{i,j},\,X]\ \text{for}\ X=Z_k,P_k,N_k,X_{q,k},Y_{q,k} \,  (q\neq i,j),
&
{[}Y_{i,j},\,X]\ \text{for}\ X=Z_k,P_k,N_k,X_{q,k},Y_{q,k} \, (q\neq i,j).
\end{array}
\]
\end{theorem}

We now turn our attention to the integral cohomology ring of $M(\CC_n)$. 
Let $\EE_n=\bigwedge[H^1(M(\CC_n))]$ be the exterior algebra 
generated by 
$z_i,\rho_i,\eta_i$ ($1\le i \le n$), $\alpha_{i,j},\beta_{i,j}$ ($1\le i<j\le n$).

\begin{theorem} \label{thm:Ccoho} The cohomology ring $H^*(M(\CC_n)) \cong \EE_n/{\II}_n$ is isomorphic to the quotient of the exterior algebra $\EE_n$ by the ideal ${\II}_n$ generated, for $1\le i,j,k\le n$ 
with $i<j$, resp., $i<j<k$, where relevant,  by
\[
\begin{array}{lll}
 z_i \rho_i,\qquad   z_i \eta_i,\qquad   \rho_i \eta_i,\\[2pt]
 ( z_j- z_i)( \alpha_{i,j}- z_i),  & ( \rho_j- \rho_i)( \alpha_{i,j}- \rho_i), &  ( \eta_j- \eta_i)( \alpha_{i,j}- \eta_i),\\
( z_j+ z_i)\beta_{i,j}, & (\rho_j+z_i-\rho_i)(\beta_{i,j}-\rho_i),& (\eta_j+z_i-\eta_i)(\beta_{i,j}-\eta_i),\\
(\alpha_{i,j}+ z_i-\rho_i-\eta_i)(\beta_{i,j}-\rho_i-\eta_i),\\[2pt]
(\alpha_{i,k}-\alpha_{i,j})(\alpha_{j,k}-\alpha_{i,j}),   
&(\alpha_{i,k}+ z_j-\beta_{i,j})(\beta_{j,k}-\beta_{i,j}),\\
(\beta_{i,k}-\beta_{i,j})(\alpha_{j,k}+z_i-\beta_{i,j}), 
&(\beta_{i,k}+ z_j-\alpha_{ij})(\beta_{j,k}+z_i-\alpha_{i,j}).
\end{array}
\]
\end{theorem}

It is readily checked that the differential forms \eqref{eq:Cforms} satisfy the relations defining the ideal $\II_n$ given in the statement of the theorem.

\begin{remark}
Since the arrangement $\CC_n$ is strictly supersolvable, by \Cref{cor:koszul}, the cohomology ring $H^*(M(\CC_n))$ is a Koszul algebra.
\end{remark}

Both Theorems \ref{thm:Cholo} and \ref{thm:Ccoho} may be established by induction.  
After some necessary preliminaries, we sketch proofs below.

For the strictly supersolvable bundle $M(\CC_{n+1}) \to M(\CC_n)$, one choice of the root map $\rootmap \colon M(\CC_n) \to \Conf_K(\C)$, where $K=2n+3$, is
\begin{equation} \label{eq:Croot}
\rootmap(x_1^{},x_2^{},\dots,x_n^{})=\bigl(0,1,-1,x_1^{},x_1^{-1},x_2^{},x_2^{-1},\dots,x_n^{},x_n^{-1}\bigr).
\end{equation}
This corresponds to ordering the coordinate hyperplane and the hypersurfaces in $\CC_{n+1}\smallsetminus\CC_n$ as follows:
\begin{equation*} \label{eq:order}
x^{}_{n+1}=0,\ x^{}_{n+1}=1,\ x^{}_{n+1}=-1,\ x^{}_{n+1}=x^{}_1,\ x^{}_{n+1}=x^{-1}_1,\dots,\  x^{}_{n+1}=x_n^{},\  x^{}_{n+1}=x^{-1}_n. 
\end{equation*}
Recall the generators 
$Z_i,P_i,N_i$ ($1\le i \le n$), $X_{i,j},Y_{i,j}$ ($1\le i<j\le n$)
for the first homology group $H_1(M(\CC_n))$ dual to the cohomology generators \eqref{eq:Cforms}, and recall that the generators of $H_1(\Conf_K(\C))=\Z^{\binom{K}{2}}$ are denoted by $A_{i,j}$, $1\le i < j \le K$.

\begin{proposition} \label{prop:homology g}
The homological root homomorphism $\rootmap_* \colon H_1(M(\CC_n)) \to H_1(\Conf_K(\C))$ is given by
\[
\begin{aligned}
\rootmap_*(Z_i)&=A_{1,2i+2}-A_{1,2i+3}-A_{2,2i+3}-\cdots -A_{2i+2,2i+3}-A_{2i+3,2i+4}-A_{2i+3,2i+5}-\cdots-A_{2i+3,K},\\
\rootmap_*(P_i)&=A_{2,2i+2}+A_{2,2i+3}+A_{2i+2,2i+3},\qquad\qquad \rootmap_*(X_{i,j})=A_{2i+2,2j+2}+A_{2i+3,2j+3},\\
\rootmap_*(N_i)&=A_{3,2i+2}+A_{3,2i+3}+A_{2i+2,2i+3},\qquad\qquad \rootmap_*(Y_{i,j})=A_{2i+2,2j+3}+A_{2i+3,2j+2}.
\end{aligned}
\]
\end{proposition}
\begin{proof}
To apply the results of \Cref{sec:hrm} in the context of the bundle $M(\CC_{n+1})\to M(\CC_n)$, we work with the basis
$\{\tilde z_i, \tilde \rho_i, \tilde \eta_i \mid 1\le i\le n\}\cup\{ 
\tilde\alpha_{i,j}, \tilde\beta_{i,j} \mid 1\le i<j\le n\}$
for the cohomology group $H^1(M(\CC_n))$, where $\tilde\alpha_{i,j}=\frac{1}{2\pi\ii}\mathrm{dlog}(x_i^{-1}x_j-1)=\alpha_{i,j}-z_i$, and 
$\tilde\upsilon=\upsilon$ for all other basis elements. Denote the dual basis for 
$H_1(M(\CC_n))$ by $\{\tilde Z_i, \tilde P_i, \tilde N_i \mid 1\le i\le n\}\cup\{ 
\tilde X_{i,j}, \tilde Y_{i,j} \mid 1\le i<j\le n\}$. \Cref{thm:hrm} yields:
\[
\begin{aligned}
\rootmap_*(\tilde Z_i)&=A_{1,2i+2}-\sum_{q=1}^{2i+2} A_{q,2i+3}
+\sum_{q=i+1}^{n} A_{2i+2,2q+2}-\sum_{q=i+1}^{n} A_{2i+3,2q+2},\\
\rootmap_*(\tilde P_i)&=A_{2,2i+2}+A_{2,2i+3}+A_{2i+2,2i+3},\qquad\qquad \rootmap_*(\tilde X_{i,j})=A_{2i+2,2j+2}+A_{2i+3,2j+3},\\
\rootmap_*(\tilde N_i)&=A_{3,2i+2}+A_{3,2i+3}+A_{2i+2,2i+3},\qquad\qquad \rootmap_*(\tilde Y_{i,j})=A_{2i+2,2j+3}+A_{2i+3,2j+2}.
\end{aligned}
\]

From the cohomology change of basis $\tilde\alpha_{i,j}=\alpha_{i,j}-z_i$ and $\tilde\upsilon=\upsilon$ for other basis elements, the elements of the desired homology basis are given by
\[
Z_i=\tilde Z_i-\sum_{j=i+1}^n \tilde X_{i,j},\quad P_i=\tilde P_i,\quad N_i=\tilde N_i,\quad X_{i,j}=\tilde X_{i,j},\quad Y_{i,j}=\tilde Y_{i,j}.
\]
With this, a straightforward calculation completes the proof.
\end{proof}

\begin{proof}[Proof sketch for Theorem \ref{thm:Cholo}]
The proof is by induction on $n$. Recall that $\lie_n$ denotes the integral lower central series Lie algebra of the fundamental group $G(\CC_n)=\pi_1(M(\CC_n))$. 

In the base case $n=1$, we have $M(\CC_1)=\C\smallsetminus\{-1,0,1\}$, $G(\CC_1)=F_3$, the free group on $3$ generators, and $\lie_1=\clie_3=\LL[H_1(M(\CC_1))]$, the free Lie algebra generated by $Z_1,P_1,N_1$. Note that in this instance, the Lie ideal $\mathcal{J}_1$ recorded in the statement of the theorem is zero.

Assuming inductively that $\lie_n \cong \clie_n/{\mathcal{J}}_n$, we must show that $\lie_{n+1} \cong \clie_{n+1}/{\mathcal{J}}_{n+1}$. 
Let $\LL[K]$ be the free Lie algebra generated by $A_{q,K+1}$, $1\le q \le K$.
From (the proof of) \Cref{thm:LCSliealg}, the Lie algebra $\lie_{n+1}$ is the semidirect product of $\lie_n$ by the free Lie algebra $\LL[K]$ determined by the Lie homomorphism $\Theta=\theta_K \circ \rootmap_* \colon \lie_n \to \Der(\LL[K])$, where $\theta_K(A_{i,j})=\ad(A_{i,j})$ and $\rootmap_*$ is induced by the root map $\rootmap \colon M(\CC_n) \to \Conf_K(\C)$. 
Recall from \eqref{eq:infbraid} that the adjoint action of the LCS Lie algebra of the pure braid group $P_{K}=\pi_1(\Conf_{K}(\C))$ on $\LL[K]$ is given on generators by
\begin{equation*}\label{eq:YBad}
\ad(A_{i,j})(A_{q,K+1})=[A_{i,j},A_{q,K+1}]=[A_{q,K+1},\,(\delta_{i,q}+\delta_{j,q})(A_{i,K+1}+A_{j,K+1})].
\end{equation*}

From the semidirect product structure of $\lie_{n+1}$, for $x\in \lie_n$ and $y\in \LL[K]$, we have 
\[
[x,y]=\Theta(x)(y) = \theta_K (\rootmap_*(x))(y)=\ad(\rootmap_*(x))(y) = [\rootmap_*(x),y]
\]
in $\LL[K] \subset \lie_{n+1}$. Using this and 
Proposition \ref{prop:homology g}, the generators $Z_i,P_i,N_i,X_{i,j},Y_{i,j}$ of $\lie_n$ and $A_{q,K+1}$ of $\LL[K]$ satisfy 
\begin{equation} \label{eq:newrelz}
\begin{aligned}
{[}Z_i,\ A_{q,K+1}]&=[A_{1,2i+2}-A_{1,2i+3}-\cdots -A_{2i+2,2i+3}-A_{2i+3,2i+4}-\cdots-A_{2i+3,K},\ A_{q,K+1}]\\
[P_i,\ A_{q,K+1}]&=[A_{2,2i+2}+A_{2,2i+3}+A_{2i+2,2i+3},\ A_{q,K+1}]\\
[N_i,\ A_{q,K+1}]&=[A_{3,2i+2}+A_{3,2i+3}+A_{2i+2,2i+3},\ A_{q,K+1}]\\
[X_{i,j},\ A_{q,K+1}]&=[A_{2i+2,2j+2}+A_{2i+3,2j+3},\ A_{q,K+1}]\\
[Y_{i,j},\ A_{q,K+1}]&=[A_{2i+2,2j+3}+A_{2i+3,2j+2},\ A_{q,K+1}]
\end{aligned}
\end{equation}

To complete the inductive proof, it suffices to show that the generators of the Lie ideal $\mathcal{J}_{n+1}$ not in $\mathcal{J}_{n}$ (i.e., those involving $Z_{n+1}$, $P_{n+1}$, $N_{n+1}$, $X_{i,n+1}$, $Y_{i,n+1}$) specified in the statement of the theorem correspond to the relations implicit in \eqref{eq:newrelz}. This may be accomplished using the infinitesimal pure braid relations \eqref{eq:infbraid} and the dictionary below.
\begin{center}
\setlength{\tabcolsep}{3pt}
\begin{tabular}{ c c c c c c c c c c c}
$A_{1,K+1}$&$A_{2,K+1}$&$A_{3,K+1}$&$A_{4,K+1}$&$A_{5,K+1}$&$\cdots$&$A_{2i+2,K+1}$&$A_{2i+3,K+1}$&$\cdots$&$A_{K-1,K+1}$&$A_{K,K+1}$\\
$Z_{n+1}$&$P_{n+1}$&$N_{n+1}$&$X_{1,n+1}$&$Y_{1,n+1}$&$\cdots$&$X_{i,n+1}$&$Y_{i,n+1}$&$\cdots$&$X_{n,n+1}$&$Y_{n,n+1}$
\end{tabular}
\end{center}
For example, we have
\[
\begin{aligned}
{[}P_i,\ A_{q,K+1}]&=[A_{2,2i+2}+A_{2,2i+3}+A_{2i+2,2i+3},\ A_{q,K+1}]\\
&=\begin{cases}
{[}A_{2,K+1},\ A_{2i+2,K+1}+A_{2i+3,K+1}]&\text{if $q=2$,}\\
{[}A_{2i+2,K+1},\ A_{2,K+1}+A_{2i+3,K+1}]&\text{if $q=2i+2$,}\\
{[}A_{2i+3,K+1},\ A_{2,K+1}+A_{2i+2,K+1}]&\text{if $q=2i+3$,}\\
0&\text{otherwise.}
\end{cases}
\end{aligned}
\]
This yields relations
\[
\begin{array}{ll}
{[}P_i+A_{2i+2,K+1}+A_{2i+3,K+1},\ A_{2,K+1}],  & {[}P_i+A_{2,K+1}+A_{2i+3,K+1},\ A_{2i+2,K+1}],\\
{[}P_i+A_{2,K+1}+A_{2i+2,K+1},\ A_{2i+3,K+1}], & {[}P_i,\ A_{q,K+1}]\quad \text{for}\ q\neq 2,2i+2,2i+3
\end{array}
\]
in $\lie_{n+1}$. Rewriting using the above dictionary, we obtain
\begin{align*}
&{[}P_i+X_{i,n+1}+Y_{i,n+1}, \ P_{n+1}], \quad {[}P_i+P_{n+1}+Y_{i,n+1},\ X_{i,n+1}],  \quad  {[}P_i+P_{n+1}+X_{i,n+1},\ Y_{i,n+1}],\\
&{[}P_i,\ Z_{n+1}],\quad [P_i,\ N_{n+1}], \quad {[}P_i,\ X_{q,n+1}], \quad {[}P_i,\ Y_{q,n+1}]\quad \text{for}\ q\neq i,
\end{align*}
which are equivalent formulations of the generators involving $P_i$ of $\mathcal{J}_{n+1}$ not in $\mathcal{J}_{n}$ in the statement of \Cref{thm:Cholo}.

The remaining generators of $\mathcal{J}_{n+1}$ not in $\mathcal{J}_{n}$ may be obtained from \eqref{eq:newrelz} in a similar manner. Details are left to the intrepid reader.
\end{proof}

\begin{proof}[Proof sketch for Theorem \ref{thm:Ccoho}]
The proof is by induction on $n$. 

In the base case $n=1$, we have $M(\CC_1)=\C\smallsetminus\{-1,0,1\}$, the exterior algebra $\EE_1$ is generated by $z_1,\rho_1,\eta_1$, and $\II_1=\langle z_1\rho_1,z_1\eta_1,\rho_1\eta_1\rangle$. Clearly, 
$H^*(M(\CC_1))\cong \EE_1/\II_1$.

Assuming inductively that $H^*(M(\CC_n))\cong \EE_n/\II_n$, to prove the theorem, it suffices to show that the generators 
$z_{n+1},\rho_{n+1},\eta_{n+1},\alpha_{i,n+1},\beta_{i,n+1}$ of  $H^*(M(\CC_{n+1}))$ satisfy the relations corresponding to the generators of $\II_{n+1}$, not in $\II_n$. As indicated in the proof of \Cref{thm:H*ring}, the defining relations of the LCS Lie algebra $\lie_{n+1}$ encode the map $\ab_*\colon H_2(M(\CC_{n+1})) \to H_2(\Z^B)$, 
where $B=(n+1)(n+3)$ is the rank of $H_1(M(\CC_{n+1}))$ and $\ab_*$ is induced by abelianization. Using \cite[Theorem 3.1]{DCalmostdirect}, we consequently need to analyze elements of the kernel of the map $\ab^*$ dual to  $\ab_*\colon H_2(M(\CC_{n+1})) \to H_2(\Z^B)$ involving classes $uv$, where $u,v\in\{z_{n+1},\rho_{n+1},\eta_{n+1},\alpha_{i,n+1},\beta_{i,n+1}\}$.

Recall from \eqref{eq:Cforms} that the cohomology generators are expressed as logarithmic differential forms, $z_i = \frac{1}{2\pi\ii}\mathrm{dlog}(x_i)$,\ldots, $\beta_{i,j} = \frac{1}{2\pi\ii}\mathrm{dlog}(x_ix_j-1)$. In the context of determining the (new) cohomology relations in $H^*(M(\CC_{n+1}))$ from the root map $\rootmap\colon M(\CC_{n}) \to \Conf_K(\C)$ of \eqref{eq:Croot}, there is a notable exception. Namely, the (``fiber'') hypersurface given by $x_ix_{n+1}-1=0$ is expressed as $x_{n+1}^{}=x_i^{-1}$, i.e., $x_{n+1}^{}-x_i^{-1}=0$, corresponding to the differential form $\frac{1}{2\pi\ii}\mathrm{dlog}(x_{n+1}^{}-x_i^{-1})$. In light of this, the aforementioned analysis should  be done in terms of classes $\hat\beta_{i,n+1}=\beta_{i,n+1}-z_i$.

We illustrate by carrying this analysis out for the class $\alpha_{i,n+1}\hat\beta_{j,n+1}$, with $i<j$. 
The relevant defining relations of the LCS Lie algebra $\lie_{n+1}$,
\[
\begin{array}{ll}
{\displaystyle{{[}z_j-z_{n+1}-\rho_{n+1}-\eta_{n+1}-\sum_{q=1}^{n} (\alpha_{q,n+1}+ \hat\beta_{q,n+1}),\,\hat\beta_{j,n+1}],}}\quad &[z_j-\hat\beta_{j,n+1},\,\alpha_{i,n+1}]\\
{[}\beta_{i,j}+\alpha_{i,n+1}+\hat\beta_{j,n+1},\,\alpha_{i,n+1}],&{[}\beta_{i,j}+\alpha_{i,n+1}+\hat\beta_{j,k},\,\hat\beta_{j,n+1}],
\end{array}
\]
yield the following element of $\ker(\ab^*)$:
\[
\alpha_{i,n+1}\hat\beta_{j,n+1}+z_j\hat\beta_{j,n+1}-z_j\alpha_{i,n+1}+\beta_{i,j}\alpha_{i,n+1}-\beta_{i,j}\hat\beta_{j,n+1}=(\alpha_{i,n+1}+z_j-\beta_{i,j})(\hat\beta_{j,n+1}-\alpha_{i,n+1}).
\]
Rewriting using $\hat\beta_{j,n+1}=\beta_{j,n+1}-z_j$ yields
\[
(\alpha_{i,n+1}+z_j-\beta_{i,j})(\beta_{j,n+1}-z_j-\alpha_{i,n+1})=
(\alpha_{i,n+1}+z_j-\beta_{i,j})(\beta_{j,n+1}-\beta_{i,j}),
\]
one of the generators of $\II_{n+1}$ not in $\II_n$.

The remaining generators of $\II_{n+1}$ not in $\II_n$ may  be obtained in a similar manner.
\end{proof}

\begin{ack}   
E.D. acknowledges the hospitality and support of the other two authors and of the Louisiana State University Department of Mathematics for a research visit where this work was initiated.

\smallskip

\noindent We thank an anonymous referee for a thorough reading of the manuscript, and for pertinent suggestions.
\end{ack}

\end{document}